\documentclass[leqno,12pt]{amsart} 
\usepackage{amssymb}
\usepackage{color} 
\theoremstyle{plain} 
\newtheorem{theorem}{\indent\sc Theorem}[section] 
\newtheorem{lemma}[theorem]{\indent\sc Lemma}
\newtheorem{corollary}[theorem]{\indent\sc Corollary}

\theoremstyle{definition} 
\newtheorem{definition}[theorem]{\indent\sc Definition}
\newtheorem{remark}[theorem]{\indent\sc Remark}

\makeatletter \@addtoreset{equation}{section} \makeatother

\newcommand{\R}{\mathbb{R}}

\DeclareMathOperator{\Ric}{Ric}

\renewcommand{\d}{\mathrm{d}}
\newcommand{\m}{\mathfrak{m} }

\newcommand{\s}{\mathfrak{s}}

\newcommand{\eps}{{\varepsilon}}

\begin{document}

\title[Laplacian comparison theorem on Riemannian manifolds]{Laplacian comparison theorem on Riemannian manifolds  with modified \boldmath$m$-Bakry-Emery Ricci lower bounds
for \boldmath$m\leq1$} 

\author[K. Kuwae]{Kazuhiro Kuwae$^*$} 

\author[T. Shukuri]{Toshiki Shukuri} 


\renewcommand{\thefootnote}{\fnsymbol{footnote}}
\footnote[0]{2020\textit{ Mathematics Subject Classification}.
 Primary 53C20; Secondary 53C21, 53C22, 53C23, 53C24, 58J60.}
\thanks{ 
$^*$Supported in part by JSPS Grant-in-Aid for Scientific Research (KAKENHI) 17H02846 and by fund (No.:185001) from the Central Research Institute of Fukuoka University.
}
\address{
Department of Applied Mathematics \endgraf
Fukuoka University \endgraf
Fukuoka 814-0180 \endgraf
Japan
}
\email{kuwae@fukukoa-u.ac.jp}

\address{
Oita City Takio Junior High School \endgraf
Oita, 874-0942 \endgraf
Japan
}
\email{lo.5.hawks61@docomo.ne.jp}

\maketitle

\begin{abstract}
In this paper, we prove a Laplacian comparison theorem  
for non-symmetric diffusion operator on complete smooth $n$-dimensional Riemannian manifold having a lower bound of modified $m$-Bakry-\'Emery Ricci tensor 
under $m\leq 1$ in terms of vector fields. 
As consequences, we give the optimal conditions for modified $m$-Bakry-\'Emery Ricci 
tensor under $m\leq1$ such that the 
(weighted) Myers' theorem, 
Bishop-Gromov volume comparison theorem,  
Ambrose-Myers' theorem, 
 Cheng's maximal diameter theorem, and the Cheeger-Gromoll type splitting theorem hold.  
Some of these results were well-studied for 
$m$-Bakry-\'Emery Ricci curvature under $m\geq n$ (\cite{Xdli:Liouville, Lot, Qian, WeiWylie}) or $m=1$ (\cite{Wylie:WarpedSplitting, WylieYeroshkin}) if the vector field is a gradient type. When $m<1$, our results are new in the literature. 
\end{abstract}

\section{Introduction} 

\subsection{Modified Bakry-\'Emery Ricci curvatures}
\ \ Let $(M,g)$ be an $n$-dimensional smooth complete Riemannian manifold with its volume measure  $\m:=\text{\rm vol}_g$  and 
$V$ a $C^1$-vector field.  
Throughout this paper, we assume that the manifold $M$ has no boundary and is connected. 
We consider a diffusion operator $\Delta_V:=\Delta- \langle V, \nabla\cdot\rangle $.  In \cite{Tadano,TadanoNegative,Wylie:WarpedSplitting},  $\Delta_V$ is called the \emph{$V$-Laplacian} on $(M, g)$. 

For any constant $m\in]-\infty,+\infty]$, we introduce the symmetric $2$-tensor 
\begin{align*}
{\Ric}_{m,n}(\Delta_V)(x)={\Ric}(x)+\frac12\mathcal{L}_Vg(x)-\frac{\;V^*(x)\otimes V^*(x)\;}{m-n},\quad x\in M,
\end{align*} 
and call it the \emph{modified $m$-Bakry-\'Emery Ricci  tensor} of the diffusion operator 
$\Delta_V$. Here $\mathcal{L}_Vg(X,Y):=\langle \nabla_XV,Y\rangle +\langle \nabla_YV,X\rangle $ is the Lie derivative of $g$ with respect to $V$ and $V^*$ is the dual $1$-form of $V$ coming from $g$.

For any $m\in]-\infty,+\infty]$ and a continuous function $K:M\to\mathbb{R}$, we call $(M, g, V)$ or $L$ satisfies the ${\rm CD}(K, m)$-condition  if 
\begin{align*}
{\rm Ric}_{m, n}(\Delta_V)(x)\geq K(x)\quad \text{ for all }\quad x\in M.
\end{align*} 
When $m=n$, we always assume that $V$ vanishes so that 
${\Ric}_{n,n}(\Delta_V)={\Ric}$. 
When $m\geq n$, $m$ is regarded as an upper bound for the dimension of the diffusion operator $\Delta_V$. 
Throughout this paper, we focus on the case $m\leq1$ and assume $n>1$ if $m=1$ and $V$ does not vanish (i.e., $V\equiv0$ and $\Delta_V=\Delta$ if $m=n=1$). 
Consequently, for $m\leq1$, we always assume $n>m$ provided $V$ does not vanish. 
Note that, for $m\leq1$, $N\in[n,+\infty[$, and for any $x\in M$,  we have 
\begin{align*}
{\Ric}_{1,n}(\Delta_V)(x)\geq {\Ric}_{m,n}(\Delta_V)(x)\geq 
{\Ric}_{\infty,n}(\Delta_V)(x)\geq {\Ric}_{N,n}(\Delta_V)(x).
\end{align*}
If we only consider the case that the lower bounds of the above Ricci tensor are constant, ${\Ric}_{1,n}(\Delta_V)\geq {\rm const}.$ is the weakest one among them. But if we consider the case that the lower bound of Ricci curvature depends on the parameter $m$ like \eqref{eq:RiciLowBddStrong} below, the similar condition is no longer the weakest one.  

In the literature, there have been intensive works on the study of geometry and analysis of weighted complete Riemannian manifolds 
with the CD$(K, m)$-condition for $m\geq n$ and $K\in \mathbb{R}$ 
(or $K\in C(M, \mathbb{R})$) in the case 
$V=\nabla\phi$ for $\phi\in C^2(M)$. See 
\cite{AN,BakryLect1581,BGL_book,BE1,BL,
BQ1,BQ2,FanLiZhang,FLL,Xdli:Liouville, Li12,LL15,LL17,Lot,Qian,WeiWylie}, 
and reference therein. During recent years,  there are already several papers on the study of weighted Riemannian manifolds with $m$-Bakry-\'Emery Ricci curvature for 
 $m<0$ or $m<1$ with $V=\nabla\phi$ for a $C^2$-function $\phi$. For $V=\nabla\phi$, we write $L:=\Delta_{\nabla\phi}$ 
in this introduction.  
 In ~\cite{OhtaTaka},   Ohta and Takatsu proved  the  $K$-displacement convexity of the R\'enyi type 
 entropy under the  $m$-Bakry-\'Emery Ricci tensor condition ${\rm Ric}_{m, n}(L)\geq K$,  i.e., the ${\rm CD}(K, m)$-condition, for $m\in]\!-\infty,0\,[\,\cup\,[\,n,+\infty\,[$ and $K\in \mathbb{R}$.  After that,  Ohta~\cite{Ohta:KN} and Kolesnikov-Milman~\cite{KolesMilman} simultaneously treated 
the case $m<0$.  Ohta~\cite{Ohta:KN} extended the Bochner inequality, eigenvalue estimates, and the Brunn-Minkowski inequality under the lower bound for 
${\Ric}_{m,n}(L)$ with $m<0$. 
Kolesnikov-Milman~\cite{KolesMilman} also proved the Poincar\'e and the Brunn-Minkowski inequalities for manifolds with boundary 
under the lower bound for 
${\Ric}_{m,n}(L)$ with $m<0$. 
In \cite[Theorem~4.10]{Ohta:KN}, Ohta also proved that the lower bound of ${\Ric}_{m,n}(L)(x)$ with 
 $m<0$ is equivalent to the curvature dimension condition in terms of 
 mass transport theory as defined by Lott-Villani~\cite{LV2} and Sturm~\cite{St:geomI, St:geomII}.  In ~\cite{Wylie:WarpedSplitting}, Wylie  proved a warped product version of Cheeger-Gromoll splitting theorem under the CD$(0, 1)$-condition. He also proved an isometric product version of Cheeger-Gromoll splitting under CD$(0, m)$-condition with $m<1$ and $(V,1)$-completeness condition. 
  In ~\cite{WylieYeroshkin}, 
W.~Wylie and D.~Yeroshkin proved a Laplacian comparison theorem, a Bishop-Gromov volume comparison theorem,  Myers' theorem and Cheng's maximal diameter theorem on manifolds with 
$m$-Bakry-\'Emery Ricci curvature condition for $m=1$ with  $V=\nabla \phi$ for a $C^2$-function $\phi$.  Recently,  Milman~\cite{Milman17} extended the Heintze-Karcher Theorem, isoperimetric inequality, and functional inequalities under the lower bound for ${\Ric}_{m,n}(L)(x)$ with $m < 1$. 
In \cite{KL}, the first named author and X.-D.~Li established 
the Laplacian comparison theorem on weighted complete Riemannian manifolds with the ${\rm CD}(K, m)$-condition with $m\leq 1$ for 
$V=\nabla\phi$ with $\phi\in C^2(M)$, and obtained (weighted) Myers' theorem, Bishop-Gromov volume comparison theorem, Ambrose-Myers' theorem, Cheeger-Gromoll type splitting theorem, stochastic completeness and Feller property of $L$-diffusion process under optimal conditions on the $m$-Bakry-\'Emery Ricci 
tensor for $m\leq1$ 
over the weighted complete Riemannian manifolds.

It is important to know whether one can establish the Laplacian comparison theorem on such Riemannian manifolds with the ${\rm CD}(K, m)$-condition for $m\leq 1$ and $K\in \mathbb{R}$ for general $C^1$-vector field $V$. In this paper, we prove such comparison theorem for $K$ being a continuous function depending on a re-parametrized distance 
 function on $M$. As consequences, we give the optimal conditions on the modified $m$-Bakry-\'Emery Ricci 
tensor for $m\leq1$ so that (weighted) Myers' theorem, Bishop-Gromov volume comparison theorem, 
Ambrose-Myers' theorem, Cheng's maximal diameter theorem, 
 and the Cheeger-Gromoll type splitting theorem hold on weighted complete Riemannian manifolds. These geometric results are complete extensions of the case for $V=\nabla \phi$ proved in the first part of \cite{KL}.  When $m<1$, our results are new in the literature.  

\par

\bigskip

\noindent
\emph{Acknowledgment.} The authors would 
thank to 
Dr.~Yohei Sakurai for his significant comments to the 
draft of this paper.  
They also would like to thank to 
the anonymous referee. 
His/Her comments help to improve the quality 
of this paper very much.


\section{Main result}
Let $V$ be a $C^1$-vector field on a Riemannian manifold $(M,g)$. Since there  may be no function $\phi$ satisfying $V=\nabla\phi$ in general, we can still make sense of bounds by integrating $V$ along geodesics.   
Define 
\begin{align*}
V_{\gamma}(r):&=\int_0^r\langle V_{\gamma_s},\dot{\gamma}_s\rangle \d s
\end{align*}
for a unit speed geodesic $\gamma:[0,T[\to M$, 
and 
\begin{align}
\phi_V(x):&=\inf\left\{\left.
\int_0^{r_p(x)}\langle V_{\gamma_s},\dot{\gamma}_s\rangle \d s
\;\;\right|\!\!\!\!\! \left.\begin{array}{ll}&\gamma: \text{unit speed geodesic}\\ & \gamma_0=p, \gamma(r_p(x))=x\end{array}\right.\right\}.\label{eq:ModifiedPhi}
\end{align} 
Note that $V_{\gamma}$ depends on the choice of unit speed geodesic $\gamma$, and 
$\phi_V(x)$ depends on $p$ with $\phi_V(p)=0$ and it is well-defined for $x\in M$. 
It is easy to see that $\phi_V(x)=\int_0^{r_p(x)}Vr_p(\gamma_s)\d s$ under $x\notin {\rm Cut}(p)$, where 
$\gamma$ is the unique unit speed geodesic with $\gamma_0=p$ and $\gamma(r_p(x))=x$. 
Hence $\phi_V$ is a continuous function on $({\rm Cut}(p)\cup\{p\})^c$.  
Consequently, $\phi_V$ is an $\m$-measurable function. 
Moreover, for $x\notin {\rm Cut}(p)$, $\phi_V(x)=V_{\gamma}(r_p(x))$ for the unique unit speed geodesic $\gamma$ with $\gamma_0=p$ and $\gamma(r_p(x))=x$. Hence $\phi_V(\gamma_t)=V_{\gamma}(t)$ for any unit speed geodesic $\gamma$ with $\gamma_0=p$ and $\gamma_t\notin {\rm Cut}(p)$.
When $V=\nabla\phi$ is a gradient vector field for some $\phi\in C^2(M)$, then one can see 
\begin{align*}
V_{\gamma}(t)&=\int_0^t\langle \nabla\phi,\dot{\gamma}_s\rangle \d s
=\int_0^t\frac{\d}{\d s}\phi(\gamma_s)\d s=\phi(\gamma_t)-\phi(\gamma_0).
\end{align*}
Throughout this paper, 
we fix a point $p\in M$ and a constant $C_p>0$, which may depend on $p$. 
For $x\in M$, we define
\begin{align*}
s_p(x):=\inf\left\{\left. C_p 
\int_0^{r_p(x)}e^{-\frac{2V_{\gamma}(t)}{n-m}}\d t\;\;\right|\left.\begin{array}{ll}&\!\!\!\!\gamma: \text{unit speed geodesic}\\ &\!\!\!\! \gamma_0=p, \gamma(r_p(x))=x\end{array}\right.\right\}.
\end{align*}

If $(M,g)$ is complete, then  $s_p(x)$ is finite  and 
well-defined from the basic properties of Riemannian geodesics. 
Let  $s(p,q):=s_p(q)$ for $p,q\in M$. If $q$ is not a cut point of $p$, then there is  a unique minimal geodesic from $p$ to $q$ and $s_p$ is smooth in a neighborhood of $q$ as 
can be computed by pulling the function back by the exponential map 
at $p$. Note that $s(p,q)\geq0$, it is zero if and only if $p=q$.  But $s(p,q)=s(q,p)$ does not hold in general. 
If $V=\nabla\phi$ for some $\phi\in C^2(M)$ and set 
$C_p=\exp\left(-\frac{\;2\phi(p)\;}{n-m} \right)$ for the definition of $s_p(x)$ with $p$ being arbitrary, then one can see that  
$s(p,q)=s(q,p)$ for $p,q\in M$.  
However, 
$s(p,q)$ does not necessarily satisfy the triangle inequality.

\begin{definition}\label{df:phicompletness}
{\rm Let $(M,g)$ be an $n$-dimensional complete Riemannian manifold and $V$ a $C^1$-vector field. Fix $p\in M$.  Then we say that $(M,g, V)$ is \emph{$(V,m)$-complete at $p$} if 
\begin{align}
\varlimsup_{r\to+\infty}\inf_{L(\gamma)=r}\int_0^re^{-\frac{\;2V_{\gamma}(t)\;}{n-m}}\d t=+\infty,\label{eq:phimcomplete}
\end{align}
where the infimum is taken over all minimizing unit speed geodesics $\gamma$ with respect to the metric $g$ such that $\gamma_0=p$. We say that  $(M,g,V)$ is \emph{$(V,m)$-complete} if it is $(V,m)$-complete at $p$ for all $p\in M$.}
\end{definition}

\begin{remark}\label{rem:VMcomplete}
{\rm  
\begin{enumerate}
\item If $V_{\gamma}$ is upper bounded for any unit speed geodesic $\gamma$ with $\gamma_0=p$,  
then $(M,g,V)$ is always $(V,m)$-complete at $p$ for all $m\leq1$. In particular, if there exists a non-negative integrable function $f$ on $[0,+\infty[$ such that $\langle V,\nabla r_p\rangle _x\leq f(r_p(x))$, then $V_{\gamma}(r)\leq\int_0^rf(t)\d t\leq\int_0^{\infty}f(t)\d t<\infty$ so that $(M,g,V)$ is always $(V,m)$-complete at $p$ for all $m\leq1$.  
\item If $M$ is compact, then $(M,g,V)$ is always $(V,m)$-complete for $m\leq1$. Indeed, 
if so, the set $G_r:=\{\gamma\mid \gamma\text{ is a unit speed minimal geodesic}, L(\gamma)=r\}$ is an empty set for sufficiently large $r > 0$. This implies \eqref{eq:phimcomplete}. 
\item If there exists a non-negative locally integrable function $f$ on $[0,+\infty[$ satisfying $f(t)\leq C/t$ on $[1,+\infty[$ for some $C\in]0,(n-m)/2]$ and $\langle V,\nabla r_p\rangle \leq f(r_p)$, then 
$(V,m)$-completeness at $p$ holds for all $m\leq1$. 
Here we assume $n>1$ for $m=1$. 
In fact, we see for $r>1$
\begin{align*}
\inf_{L(\gamma)=r}\int_0^re^{-\frac{\;2V_{\gamma}(t)\;}{n-m}}\d t 
&\geq 
\int_1^re^{-\frac{\;2\int_0^tf(s)\d s\;}{n-m}}\d t\\
&\geq  e^{-\frac{\;2\int_0^1f(s)\d s\;}{n-m}}
\int_1^r e^{-\frac{\;2\int_1^tf(s)\d s\;}{n-m}}\d t\\
&\geq e^{-\frac{\;2\int_0^1f(s)\d s\;}{n-m}}
\int_1^r e^{-\frac{\;2C\log t\;}{n-m}}\d t\\
&\geq e^{-\frac{\;2\int_0^1f(s)\d s\;}{n-m}}
\int_1^r\frac{\d t}{\;t^{\frac{\;2C\;}{n-m}}\;}\to+\infty\quad \text{ as }\quad r\to\infty,
\end{align*}
where the infimum is taken over all minimizing unit speed geodesics $\gamma$ with $\gamma_0=p$.
\item The $(V,1)$-completeness at $p$ defined as in \cite[Definition~6.2]{Wylie:WarpedSplitting} implies the $(V,m)$-completeness at $p$ for every $m\leq1$   
provided $V_{\gamma}\geq0$ 
 for any unit speed geodesic $\gamma$ with $\gamma_0=p$. 
The converse also holds under $V_{\gamma}\leq0$ 
for any unit speed geodesic $\gamma$ with $\gamma_0=p$.   
\end{enumerate}
}
\end{remark}

\begin{lemma}\label{lem:phimcomplete}
Let $(M,g)$ be an $n$-dimensional complete non-compact Riemannian manifold and $V$ a $C^1$-vector field. Fix $p\in M$ and suppose that 
$(M,g,V)$ is $(V,m)$-complete at $p$. Then, for any sequence $\{q_i\}$ in $M$  such that  $d(p,q_i)\to+\infty$ as $i\to+\infty$, 
$s(p,q_{i})\to+\infty$. 
\end{lemma}
\begin{proof}
The proof is similar to that of \cite[Proposition~3.4]{WylieYeroshkin}. We omit it.   
\end{proof}

\begin{remark}\label{rem:SpRp}
{\rm  
Recall that $\phi_V$ depends on $p\in M$. 
For a fixed $p\in M$, we set  
$\underline{\phi}_V(r):=\inf_{B_r(p)}\phi_V$ and 
$\overline{\phi}_V(r):=\sup_{B_r(p)}\phi_V$
for $r\in]0,+\infty[$. Then $\underline{\phi}_V(r)\leq0\leq 
\overline{\phi}_V(r)$ for 
$r>0$ 
and $\lim_{r\to0}\overline{\phi}_V(r)=\lim_{r\to0}\underline{\phi}_V(r)=0$. 
If $x\notin {\rm Cut}(p)$, we have 
$s_p(x)= C_p 
 \int_0^{r_p(x)}e^{-\frac{\;2\phi_V(\gamma_t)\;}{n-m}}\d t$ for the unique unit speed geodesic $\gamma$ with $\gamma_0=p$ and $\gamma(r_p(x))=x$. So $\lim_{x\to p}\frac{s_p(x)}{r_p(x)}= C_p$. 
In particular, 
\begin{align*}
 C_p e^{-\frac{\;2\overline{\phi}_V(r_p(x))\;}{n-m}}r_p(x)\leq s_p(x)\leq  C_p 
e^{-\frac{\;2\underline{\phi}_V(r_p(x))\;}{n-m}}r_p(x)\quad\text{ for }\quad x\notin {\rm Cut}(p).
\end{align*} 

} 
\end{remark}

\subsection{Laplacian Comparison} 
Let $\kappa:[0,+\infty[\to\R$ be a continuous function and ${\sf a}_{\kappa}$ the unique solution defined on the maximal interval $]0,\delta_{\kappa}[$ for $\delta_{\kappa}\in]0,+\infty]$ of the following Riccati equation 
\begin{align}
-\frac{\d {\sf a}_{\kappa}}{\d s}(s)=\kappa(s)+{\sf a}_{\kappa}(s)^2\label{eq:RiccatiEq}
\end{align}
with the boundary conditions 
\begin{align}
\lim_{s\downarrow 0}s\, {\sf a}_{\kappa}(s)=1,\label{eq:BdryCond}
\end{align}
and 
\begin{align}
\lim_{s\uparrow \delta_{\kappa}}(s-\delta_{\kappa})\, {\sf a}_{\kappa}(s)=1\label{eq:BdryCondStrict*}
\end{align}
under $\delta_{\kappa}<\infty$.  
\eqref{eq:BdryCond} yields 
\begin{align}
\lim_{s\downarrow0}{\sf a}_{\kappa}(s)=+\infty.\label{eq:BdryCond0}
\end{align}
If $\delta_{\kappa}<\infty$, from \eqref{eq:BdryCondStrict*}, $\delta_{\kappa}$ is the explosion time of ${\sf a}_{\kappa}$ in the sense that 
\begin{align}
\lim_{s\uparrow\delta_{\kappa}}{\sf a}_{\kappa}(s)=-\infty.\label{eq:BdryCond*}
\end{align}
Actually, ${\sf a}_{\kappa}(s)={\s_{\kappa}'(s)}/{\s_{\kappa}(s)}$, where $\s_{\kappa}$ is the unique solution of 
Jacobi equation $\s_{\kappa}''(s)+\kappa(s)\s_{\kappa}(s)=0$ with $\s_{\kappa}(0)=0$, $\s_{\kappa}'(0)=1$, and 
$\delta_{\kappa}=\inf\{s>0\mid \s_{\kappa}(s)=0\}$. 
We write ${\sf a}_{\kappa}(s)=\cot_{\kappa}(s)$. 
Moreover, $]0,\delta_{\kappa}[\ni s\mapsto \cot_{\kappa}(s)$ is decreasing (resp.~strictly decreasing) provided $\kappa(s)$ is non-negative (resp.~positive) for all $s\in]0,\delta_{\kappa}[$ in view of 
 \eqref{eq:RiccatiEq}.   
If $\kappa$ is a real constant, then 
\begin{align*}
{\sf a}_{\kappa}(s)=\left\{\begin{array}{cc}\sqrt{\kappa}\cot(\sqrt{\kappa}s) & \kappa>0, \\ 1/s & \kappa=0, \\ \sqrt{-\kappa}\coth(\sqrt{-\kappa}s) & \kappa<0\end{array}\right.
\end{align*}
and $\delta_{\kappa}=\pi/\sqrt{\kappa^+}\leq+\infty$. 
Fix $m\in ]-\infty,1\,]$  
and set $m_{\kappa}(s):=(n-m)\cot_{\kappa}(s)$. Then \eqref{eq:RiccatiEq} is equivalent to 
\begin{align}
-\frac{\d m_{\kappa}}{\d s}(s)=(n-m)\kappa(s)+\frac{\;m_{\kappa}(s)^2\;}{n-m},\label{eq:RiccatiEqM}
\end{align}
and \eqref{eq:BdryCond} (resp.~\eqref{eq:BdryCondStrict*}) is equivalent to 
$\lim_{s\downarrow 0}s\, m_{\kappa}(s)=n-m$ (resp.~$\lim_{s\uparrow \delta_{\kappa}}(s-\delta_{\kappa})\, m_{\kappa}(s)=n-m$ under $\delta_{\kappa}<\infty$). 
In view of the uniqueness of the solution to  \eqref{eq:RiccatiEq} with \eqref{eq:BdryCond0},  
we have the scaling property ${\sf a}_{\kappa_{\alpha}}(s)=\frac{1}{\alpha}{\sf a}_{\alpha^2\kappa}(s/\alpha)$ for $\alpha>0$. Here $\kappa_{\alpha}(s):=\kappa(s/\alpha)$. In particular, ${\sf a}_{\kappa}(s)=\frac{1}{\alpha}{\sf a}_{\alpha^2\kappa}(s/\alpha)$ for $\alpha>0$ provided $\kappa$ is a constant. 
\vspace{0.5cm}

Our first result is the following Laplacian comparison along unit speed geodesic on 
weighted  complete Riemannian manifolds $(M,g,V)$ under the 
lower bound of modified $m$-Bakry-\'Emery Ricci tensor 
 for $m\leq1$. 

\begin{theorem}[Laplacian Comparison Theorem]\label{thm:GlobalLapComp}
Suppose that $(M,g)$ is an $n$-dimen\-sional complete smooth Riemannian manifold and $V$ is a $C^1$-vector field. Fix $p\in M$. 
Take $R\in]0,+\infty]$. 
Let $\phi_V$ be the function defined in \eqref{eq:ModifiedPhi}.  
Suppose that 
\begin{align}
{\Ric}_{m,n}(\Delta_V)_x(\nabla r_p,\nabla r_p)\geq(n-m)\kappa(s_p(x)) e^{-\frac{\;4\phi_V(x)\;}{n-m}} C_p^2  \label{eq:RiciLowBdd}
\end{align} 
holds under $r_p(x)<R$ with $x\in ({\rm Cut}(p)\cup\{p\})^c$. Then  
\begin{align}
(\Delta_V r_p)(x)\leq (n-m)\cot_{\kappa}(s_p(x))e^{-\frac{\;2\phi_V(x)\;}{n-m}}  C_p. \label{eq:GloLapComp}
\end{align}
\end{theorem}

\begin{corollary}\label{cor:GlobalLapComp}
Suppose that $(M,g)$ is an $n$-dimensional complete smooth Riemannian manifold and $V$ is a $C^1$-vector field.  
Fix $p\in M$ and assume $\delta_{\kappa}<\infty$. Then 
\begin{align*}
\lim_{s_p(x)\uparrow \delta_{\kappa}}\Delta_V r_p(x)=-\infty.
\end{align*}
\end{corollary}

\begin{remark}
{\rm The sufficient condition \eqref{eq:RiciLowBdd} under $r_p(x)<R$ with $x\in ({\rm Cut}(p)\cup\{p\})^c$ 
for our 
Laplacian 
comparison theorem is weaker than the condition: 
\begin{align}
{\Ric}_{m,n}(\Delta_V)(x)\geq(n-m)\kappa(s_p(x)) e^{-\frac{\;4\phi_V(x)\;}{n-m}}  C_p^2 \;g_x
\label{eq:RiciLowBddStrong}
\end{align} 
under $r_p(x)<R$, because $\nabla r_p(x)$ is defined only for 
$x\notin{\rm Cut}(p)\cup\{p\}$. In particular, CD$(K,m)$-condition for 
$K(x)=(n-m)\kappa(s_p(x)) e^{-\frac{\;4\phi_V(x)\;}{n-m}}  C_p^2 $ always implies that \eqref{eq:RiciLowBdd} holds for all 
$x\in ({\rm Cut}(p)\cup\{p\})^c$.  
}
\end{remark}

\begin{remark}
{\rm The inequality \eqref{eq:GloLapComp} is meaningful at $p$, because 
$m_{\kappa}(0+)=+\infty$ and $\Delta_Vr_p(p)=\Delta r_p(p)=+\infty$ in view of the classical Laplacian comparison theorem for $\Delta$ under local upper sectional curvature bound (see \cite[Theorem~3.4.2]{Hsu:2001}). Moreover, the following inequality 
\begin{align}
r_p(x)(\Delta_V r_p)(x)\leq (n-m)r_p(x)\cot_{\kappa}(s_p(x))e^{-\frac{\;2\phi_V(x)\;}{n-m}}C_p\label{eq:GloLapComp*}
\end{align}
is also meaningful at $p$. Indeed, the right hand side of 
\eqref{eq:GloLapComp*} has the value $n-m$ at $x=p$ by Remark~\ref{rem:SpRp} and the left hand side has the value $n-1$ at $x=p$ by the classical Laplacian comparison theorem for $\Delta$ as noted above. 
}
\end{remark}

\begin{remark}
{\rm Theorem~\ref{thm:GlobalLapComp} 
generalizes \cite[Theorem~4.4]{WylieYeroshkin}. 
}
\end{remark}
\subsection{Geometric consequences}

\begin{theorem}[Weighted Myers' Theorem]\label{thm:WeightedMyers}
Let $(M,g)$  be an $n$-dimensional complete Riemannian manifold and a $C^1$-vector field $V$.  Fix $p\in M$.  
Assume that \eqref{eq:RiciLowBdd} 
holds for all $x\in ({\rm Cut}(p)\cup\{p\})^c$ 
 and $\delta_{\kappa}<\infty$. 
 Then $s(p,q)\leq \delta_{\kappa}$ for all $q\in M$. 
\end{theorem}

\begin{corollary}\label{cor:WeightedMyers}
Let $(M,g)$  be an $n$-dimensional complete Riemannian manifold and a $C^1$-vector field $V$.  Fix $p\in M$  and $\delta_{\kappa}<\infty$. Assume that \eqref{eq:RiciLowBdd} 
holds for all $x\in ({\rm Cut}(p)\cup\{p\})^c$ and $(M,g,V)$ is $(V,m)$-complete at $p$.   Then $M$ is compact. 
\end{corollary}
\begin{remark}
{\rm 
\begin{enumerate}
\item Theorem~\ref{thm:WeightedMyers} (resp.~Corollary~\ref{cor:WeightedMyers}) generalizes \cite[Theorem~2.2]{WylieYeroshkin} (resp. \cite[Corollary~2.3]{WylieYeroshkin}).
\item Since $V_{\gamma}\leq0$ for any unit speed geodesic $\gamma$ with $\gamma_0=p$
 implies the $(V,m)$-completeness at $p$, Corollary~\ref{cor:WeightedMyers} implies the compactness of $M$ 
 provided $\delta_{\kappa}<\infty$, \eqref{eq:RiciLowBdd} 
holds for $x\notin{\rm Cut}(p)\cup\{p\}$ and $V_{\gamma}\leq0$ 
any unit speed geodesic $\gamma$ with $\gamma_0=p$.  
\end{enumerate}
}
\end{remark}

Based on Theorems~\ref{thm:GlobalLapComp} and \ref{thm:WeightedMyers}, we can deduce several geometric 
fruitful results. Next  
we will give two versions of the Bishop-Gromov type volume comparison. 
The first one is for $\mu_V(A)=\int_Ae^{-\phi_V(x)}\m(\d x)$ of metric annuli $A(p,r_0,r_1):=\{x\in M\mid r_0\leq r_p(x)\leq r_1\}$. The comparison in this case will be in terms of the quantities
\begin{align}
\overline{\nu}_p(\kappa,r_0,r_1)&:=\int_{r_0}^{r_1}\int_{\mathbb{S}^{n-1}}\s_{\kappa}^{n-m}\left(\sup_{\eta}s_p(r,\eta) 
\right)\d r\d\theta,
\qquad \overline{\nu}_p(\kappa,r_1):=\overline{\nu}_p(\kappa,0,r_1),
\label{eq:VolSpaceFormNormalUpper}\\
\underline{\nu}_p(\kappa,r_0,r_1)&:=\int_{r_0}^{r_1}\int_{\mathbb{S}^{n-1}}\s_{\kappa}^{n-m}\left(\inf_{\eta}s_p(r,\eta)
\right)\d r\d\theta,
\qquad \underline{\nu}_p(\kappa,r_1):=\underline{\nu}_p(\kappa,0,r_1),
\label{eq:VolSpaceFormNormalLower}\\
\nu_p(\kappa,r_0,r_1)&:=\int_{r_0}^{r_1}\int_{\mathbb{S}^{n-1}}\s_{\kappa}^{n-m}(s_p(r,\theta))\d r\d\theta,
\qquad \nu_p(\kappa,r_1):=\nu_p(\kappa,0,r_1)
\label{eq:VolSpaceFormNormal}
\end{align}
under $s_p(r_1,\theta)\leq \delta_{\kappa}$ for all $\theta\in\mathbb{S}^{n-1}$. Here 
\begin{align*}
s_p(r,\theta):=C_p
\int_0^re^{-\frac{\;2V_\gamma(t)\;}{n-m}}\d t
\end{align*}
with $\theta=\dot{\gamma}_0$, and $\overline{\phi}_V(r)$ and $\underline{\phi}_V(r)$ are the functions defined in Remark~\ref{rem:SpRp}.  
If $\phi_V$ is rotationally symmetric around $p$, i.e., if there exists a $C^2$-function ${\Phi}_V$ on $[0,+\infty[$ such that $\phi_V(x)={\Phi}_V(r_p(x))$,  then $s_p(r,\theta)$ is independent of $\theta\in \mathbb{S}^{n-1}$.  
The second one  is for $\nu_V(A):=\int_Ae^{-\frac{\;2\phi_V(x)\;}{n-m}}\mu_V(\d x)=\int_Ae^{-\frac{\;n-m+2\;}{n-m}\phi_V(x)}\m(\d x)$ of the sets $C(p,s_0,s_1):=\{x\in M\mid s_0\leq s_p(x)\leq s_1\}$ and $C_s(p):= C(p,0,s)$. 
The set $C(p,s_0,s_1)$ also depends on $s_p$ and is quite different from annuli. 
The comparison in this case will be in terms of the 
quantities  
\begin{align}
v(\kappa,s_0,s_1):=\int_{s_0}^{s_1}\int_{\mathbb{S}^{n-1}}\s_{\kappa}^{n-m}(s)\d s\d\theta \quad \text{ and }\quad 
v(\kappa,s_1):=v(\kappa,0,s_1)
\label{eq:VolSpaceForm}
\end{align}
under $s_1\leq \delta_{\kappa}$. 
When $m\in ]-\infty,1]$ is an integer and $\kappa$ is a constant, \eqref{eq:VolSpaceForm} is the volume of 
annuli in the simply connected space form of constant curvature $\kappa$ and  dimension 
$n-m+1$.

\begin{theorem}[Bishop-Gromov Volume Comparison]\label{thm:BGVol}
Fix $p\in M$ and $R\in]0,+\infty]$. 
Suppose that $(M,g)$ is an $n$-dimensional complete smooth Riemannian manifold and a $C^1$-vector field $V$. 
Let $\kappa:[0,+\infty[\to\R$ be a continuous function.  Assume that \eqref{eq:RiciLowBdd} 
holds for $r_p(x)<R$ with $x\in ({\rm Cut}(p)\cup\{p\})^c$. Then we have the following:
\begin{enumerate}
\item\label{item:BG1} 
Suppose that $0\leq r_0< r_a\leq r_1$ and $0\leq r_0\leq r_b<r_1$. Then
\begin{align}
\frac{\;\mu_V(A(p,r_b,r_1))\;}{\mu_V(A(p,r_0,r_a))}\leq \frac{\;\overline{\nu}_p(\kappa,r_b,r_1)\;}{\underline{\nu}_p(\kappa,r_0,r_a)}\label{eq:BGAnnuliUpLow}
\end{align}
holds for $r_1<R$.  
Assume further that $\phi_V$ is rotationally symmetric around $p$. Then 
\begin{align}
\frac{\;\mu_V(A(p,r_b,r_1))\;}{\mu_V(A(p,r_0,r_a))}\leq \frac{\;\nu_p(\kappa,r_b,r_1)\;}{\nu_p(\kappa,r_0,r_a)}\label{eq:BGAnnuli}
\end{align}
holds for $r_1<R$, in particular, the function 
\begin{align}
]0,R[\, \ni \,r\mapsto \frac{\;\mu_V(B_r(p))\;}{\nu_p(\kappa,r)}\label{eq:BG}
\end{align}
is non-increasing. 
\item\label{item:BG2} 
Suppose that $0\leq s_0<s_a\leq s_1$ and $0\leq s_0\leq s_b<s_1$. Then
\begin{align}
\frac{\;\nu_V(C(p,s_b,s_1))\;}{\nu_V(C(p,s_0,s_a))}\leq \frac{\;v(\kappa,s_b,s_1)\;}{v(\kappa,s_0,s_a)}\label{eq:BGAnnuli*}
\end{align}
holds for $s_1<S$. 
In particular, the function 
\begin{align}
]0,S[\, \ni \, s\mapsto \frac{\;\nu_V(C_s(p))\;}{v(\kappa,s)}\label{eq:BG*}
\end{align}
is non-increasing. Here $S=\inf_{\theta\in\mathbb{S}^{n-1}}s_p(R,\theta)$.  
\end{enumerate}
\end{theorem}
\begin{remark}
{\rm \eqref{eq:BG} (resp.~\eqref{eq:BG*}) may not be bounded as $r\to0$ (resp.~$s\to0$) unless $m=1$. Note that the Bishop type inequality holds for $m=1$ (see 
\cite[Corollary~4.6]{WylieYeroshkin}).  
}
\end{remark}

\begin{corollary}\label{cor:BGVol}
Fix $p\in M$ and $R\in]0,+\infty]$. 
Suppose that $(M,g)$ is an $n$-dimensional complete smooth Riemannian manifold and $V$ is a $C^1$-vector field. 
Assume that 
\begin{align*}
{\rm Ric}_{m,n}(\Delta_V)_x(\nabla r_p,\nabla r_p)\geq 0
\quad\text{ for }\quad r_p(x)<R\quad\text{ with }\quad x\notin {\rm Cut}(p)\cup\{p\}.
\end{align*}
Then
\begin{align}
\frac{\;\mu_V(B_{r_2}(p))\;}{\mu_V(B_{r_1}(p))}\leq 
e^{2(\overline{\phi}_V(r_1)-\underline{\phi}_V(r_2))}\left(\frac{\;r_2\;}{r_1}\right)^{n-m+1}\quad\text{ for all }\quad 0<r_1<r_2<R\label{eq:BGBall}
\end{align}
holds. 
\end{corollary}

\begin{theorem}[Ambrose-Myers' Theorem]\label{thm:AmbroseMyers}
Let $(M,g)$  be an $n$-dimensional complete Riemannian manifold and $V$ a $C^1$-vector field. Fix $p\in M$.  
Assume that $(M,g,V)$ is $(V,m)$-complete at $p$. 
Suppose that  for every unit speed $($local minimizing$)$ geodesic $\gamma$ 
with $\gamma_0=p$, we have  
\begin{align}
\int_0^{\infty}e^{\frac{\;2V_{\gamma}(t)\;}{n-m}}{\Ric}_{m,n}(\Delta_V)(\dot{\gamma}_t,\dot{\gamma}_t)\d t=+\infty.\label{eq:Ambrose}
\end{align}
Then $M$ is compact. 
\end{theorem}
\begin{corollary}\label{cor:AmbroseMyersVbdd}
Let $(M,g)$  be an $n$-dimensional complete Riemannian manifold and $V$ a $C^1$-vector field. Fix $p\in M$. 
Assume ${\Ric}_{m,n}(\Delta_V)\geq0$ on $M$. 
Suppose that  there exists a non-negative measurable function $f$ on $[0,+\infty[$ satisfying $\int_0^{\infty}f(s)\d s<+\infty$ and $\langle V,\nabla r_p\rangle \geq -f(r_p)$, and 
for every unit speed $($local minimizing$)$ geodesic $\gamma$ 
with $\gamma_0=p$, we have  
\begin{align}
\int_0^{\infty}{\Ric}_{m,n}(\Delta_V)(\dot{\gamma}_t,\dot{\gamma}_t)\d t=+\infty.\label{eq:Ambrose*}
\end{align} 
Then $M$ is compact. 
\end{corollary}
\begin{corollary}\label{cor:AmbroseMyers}
Let $(M,g)$  be an $n$-dimensional complete Riemannian manifold and $V$ a $C^1$-vector field. Fix 
$p\in M$ and a constant $\kappa>0$.   
Assume that $(M,g,V)$ is $(V,m)$-complete at $p$. Suppose that for every unit speed $($local minimizing$)$ geodesic $\gamma$ with $\gamma_0=p$, we have 
\begin{align}
{\Ric}_{m,n}(\Delta_V)(\dot{\gamma}_t,\dot{\gamma}_t)\geq (n-m)\kappa e^{-\frac{\;4V_{\gamma}(t)\;}{n-m}}  C_p^2. \label{eq:AmbroseRiccibounds} 
\end{align} 
Then $M$ is compact. 
\end{corollary}

\begin{remark}
{\rm 
\begin{enumerate}
\item 
Theorem~\ref{thm:AmbroseMyers} is a version of Ambrose's Theorem (\cite{Ambrose}). Here  Ambrose's Theorem states that if for any (local minimizing) geodesic $\gamma$ emanating from a point $p\in M$, 
 $$
\int_0^{\infty}{\Ric}(\dot{\gamma}_t,\dot{\gamma}_t)\d t=+\infty, 
$$
then $M$ is compact. 
Cavalcante-Oliveira-Santos \cite{CavOliSantos} also proved the following different 
version of Ambrose's Theorem (see \cite[Theorem~2.1]{CavOliSantos}): 
Suppose that  every (local minimizing) geodesic $\gamma$ emanating from $p$ satisfies 
$$
\int_0^{\infty}{\Ric}_{m,n}(\Delta_{\nabla\phi})(\dot{\gamma}_t,\dot{\gamma}_t)\d t=+\infty
$$
under $m>n$  for $\phi\in C^2(M)$. Then $M$ is compact. Tadano~\cite[Theorem~14]{Tadano} extends  \cite[Theorem~2.1]{CavOliSantos} for $\Delta_V$ with modified $m$-Bakry-\'Emery Ricci tensor ${\rm Ric}_{m,n}(\Delta_V)$ under $m>n$. 
Our Theorem~\ref{thm:AmbroseMyers} is different  
from the above mentioned results. 
Tadano~\cite[Theorem~25]{TadanoNegative} also proves a version of Ambrose's Theorem for 
$\Delta_V$ with modified $1$-Bakry-\'Emery Ricci tensor 
${\rm Ric}_{1,n}(\Delta_V)$ 
under the condition ${\rm Ric}_{1,n}(\Delta_V)>0$ and $|V|\leq ke^{-\ell r_p}$ for some $k\geq0,\ell>0$. So our condition in Corollary~\ref{cor:AmbroseMyersVbdd} is milder than one in \cite[Theorem~25]{TadanoNegative}. 
\end{enumerate}
}
\end{remark}

In the following theorem and its corollary, we assume $V=\nabla\phi$ for some 
$\phi\in C^2(M)$ and set $C_p
=\exp\left(-\frac{\;2\phi(p)\;}{n-m} \right)$ for the definition of $s_p(x)$ with $p$ being an arbitrary point. As noted before, $s(p,q)$ is symmetric for any $p,q\in M$.  
Let 
$h=e^{-\frac{\;4\phi\;}{n-m}}g$ be the conformal change of the metric $g$. Then $s(p,q)$ is the smallest length in the $h$ metric of a minimal geodesic 
between $p$ and $q$ in the $g$ metric. As such, 
$d^{\,h}(p,q)\leq s(p,q)$  for any $q\in M$. So Theorem~\ref{thm:WeightedMyers} tells
us that the diameter of the metric $h$ is less than or equal to $\delta_{\kappa}$. For this conformal 
diameter estimate we also obtain the following rigidity characterization.
\begin{theorem}[Cheng's Maximal Diameter Theorem]\label{thm:ChengDiamSphere} 
Suppose that $(M,g)$, $n> 1$, is a complete Riemannian manifold and  $\phi\in C^2(M)$.  Fix $p,q\in M$. 
Assume that $\delta_{\kappa}<\infty$, $\kappa$ is positive on $]0,\delta_{\kappa}[$, $\kappa(s)=\kappa(\delta_{\kappa}-s)$ for all $s\in[0,\delta_{\kappa}]$, and 
\eqref{eq:RiciLowBdd} 
holds for all $x\in ({\rm Cut}(p)\cup\{p\})^c$. 
We further assume that \eqref{eq:RiciLowBdd} by  replacing $p$ with $q$ 
holds for all $x\in ({\rm Cut}(q)\cup\{q\})^c$. 
If  $d^{\,h}(p,q)=
\delta_{\kappa}$, 
then $m=1$, 
$\phi$ is rotationally symmetric around $p$, i.e., $\phi$ is a function depending only on radial $r$,    
and 
$g$ is a warped product metric of the form 
\begin{align*}
g&=\d r^2+e^{\frac{\;2 \phi(r) +2\phi(0)\;}{n-1}}
\s_{\kappa}^2(s(r))
g_{\mathbb{S}^{n-1}}, \quad 0\leq r\leq d(p,q),
\end{align*}
where $s(r)=\int_0^re^{-\frac{\;2 \phi (t)\;}{n-1}}\d t$ 
and $s(d(p,q))=
\delta_{\kappa}$. 
\end{theorem}
\begin{corollary}\label{cor:ChengDiamSphere}
Suppose that $(M,g)$, $n> 1$, is  a complete Riemannian manifold and $\phi\in C^2(M)$. Fix $p,q\in M$.  
Assume that $\kappa$ is a positive constant and 
\eqref{eq:RiciLowBdd} 
holds for all $x\in ({\rm Cut}(p)\cup\{p\})^c$. 
We further assume that \eqref{eq:RiciLowBdd} by  replacing $p$ with $q$ 
holds for all $x\in ({\rm Cut}(q)\cup\{q\})^c$. 
If  $d^{\,h}(p,q)=
\pi/\sqrt{\kappa}$, 
then $m=1$,  
$\phi$ is rotationally symmetric around $p$, i.e., $\phi$ is a function depending only on radial $r$, 
and 
$g$ is a  warped product metric of the form 
\begin{align*}
g&=\d r^2+e^{\frac{\;2 \phi(r) +2\phi(0)\;}{n-1}}\cdot
\frac{\;\sin^2({\sqrt\kappa}(s(r)))\;
}{\kappa}
g_{\mathbb{S}^{n-1}}, \quad 0\leq r\leq d(p,q),
\end{align*}
where $s(r)=\int_0^re^{-\frac{\;2 \phi(t)\;}{n-1}}\d t$ 
and $s(d(p,q))=
\pi/\sqrt{\kappa}$.  
\end{corollary}

\begin{theorem}[Cheeger-Gromoll Splitting Theorem]\label{thm:Splitting}
Let $(M,g)$ be an $n$-dimen\-sional non-compact complete Riemannian manifold and $V$ a $C^1$-vector field. 
Suppose that $(M,g,V)$ is $(V,m)$-complete  
and $M$ contains a line. 
Then under ${\rm CD}(0,m)$-condition
with $m<1$, $M$ is isometric to $\R\times N$ and $V$ depends only on $N$.
\end{theorem}

\begin{corollary}\label{cor:Splitting}
Let $(M,g)$ be an $n$-dimensional non-compact complete Riemannian manifold and $C^1$-vector field  $V$.
Suppose that $V_{\gamma}\leq0$ 
for any unit speed geodesic $\gamma$ 
and $M$ contains a line. 
Then under ${\rm CD}(0,m)$-condition
with $m<1$, we have that $M$ is isometric to $\R\times N$ and $V$ depends only on $N$.
\end{corollary}
\begin{proof}
If $V_{\gamma}\leq0$ 
for any unit speed geodesic $\gamma$, then $(M,g,V)$ is $(V,m)$-complete for all $m\leq1$. 
So the assertion easily follows Theorem~\ref{thm:Splitting}. 
\end{proof}

\begin{remark}
{\rm Theorem~\ref{thm:Splitting} partially extends  \cite[Corollary~6.7]{Wylie:WarpedSplitting} for a restricted case, where 
${\rm CD}(0,m)$-condition 
for $m<1$ and 
$(V,1)$-completeness of $(M,g,V)$ are assumed  
for the isometric splitting $M=\R\times N$. 
Note that 
the $(V,m)$-completeness does not necessarily mean 
the $(V,1)$-completeness, and it 
is weaker than $(V,1)$-completeness if $V_{\gamma}\geq0$ for any unit speed geodesic $\gamma$.  
} 
\end{remark}

\section{Proof of Theorem~\ref{thm:GlobalLapComp}}
Recall the $V$-Laplacian $\Delta_Vu:=\Delta u-\langle V,\nabla u\rangle $.  Letting $\lambda(r,\theta)= C_p^{-1} 
e^{\frac{\;2V_{\gamma}(r)\;}{n-m}}\Delta_Vr_p(r,\theta)$, we find that 
$\lambda$ satisfies the Riccati differential inequality in terms of the parameter $s$.

\begin{lemma}\label{lem:RiccatiDiferIneq}
Let $\gamma$ be a unit speed minimal geodesic with $\gamma_0=p$ and $\dot{\gamma}_0=\theta$. 
Let $s$ be the 
parameter $\d s= C_p e^{-\frac{\;2V_{\gamma}(r)\;}{n-m}}\d r$. Then 
\begin{align}
\frac{\d\lambda}{\d s}\leq -\frac{\lambda^2}{n-m}-
 C_p^{-2} e^{\frac{\;4V_{\gamma}(r)\;}{n-m}}\text{
\rm Ric}_{m,n}(\Delta_V)\left(\dot{\gamma}_r,\dot{\gamma}_r \right)\label{eq:lambda/ds}
\end{align}
in particular, 
\begin{align}
\frac{\d\lambda}{\d r}\leq -  C_p e^{-\frac{\;2V_{\gamma}(r)\;}{n-m}}
\frac{\;\lambda^2\;}{n-m}- C_p^{-1} e^{\frac{\;2V_{\gamma}(r)\;}{n-m}}\text{
\rm Ric}_{m,n}(\Delta_V)\left(\dot{\gamma}_r,
\dot{\gamma}_r \right)\label{eq:lambda/dr}
\end{align}
holds for $x=(r,\theta)\notin \text{\rm Cut}(p)\cup\{p\}$. 
Moreover, if equality is achieved at a point, then $m=1$ and at that point $\nabla_{\nabla r_p}$ has at most one non-zero eigenvalue which is of multiplicity $n-1$. 
\end{lemma}
\begin{proof} We modify the proof of the Laplacian comparison theorem on weighted complete Riemannian manifolds with the CD$(K, 1)$-condition by Wylie and Yeroshkin \cite{WylieYeroshkin}. 
The usual Bochner-Weitzenb\"ock formula for functions says that for any $u\in C^3(M)$,
\begin{align*}
\frac12\Delta |\nabla u|^2=|\nabla^2\,u|^2+{\Ric}(\nabla u,\nabla u)+\langle \nabla\Delta u,\nabla u\rangle .
\end{align*}
The Bochner-Weitzenb\"ock formula for the $V$-Laplacian and the $m$-Bakry-\'Emery 
Ricci curvature is given by 
\begin{align*}
\frac12 \Delta_V|\nabla u|^2&=|\nabla^2\,u|^2+{\Ric}_{\infty,n}(\Delta_V)(\nabla u,\nabla u)+\langle \nabla \Delta_V u,\nabla u\rangle \\ 
&=|\nabla^2\,u|^2+{\Ric}_{m,n}(\Delta_V)(\nabla u,\nabla u)-\frac{\;V^*\otimes V^*\;}{n-m}(\nabla u,\nabla u)+\langle \nabla \Delta_V u,\nabla u\rangle .
\end{align*}
Consider this equation with $u=r_p$ at an interior point of a minimizing geodesic 
(so that $r_p$ is smooth in a neighborhood). Then $|\nabla r_p|=1$ in this neighborhood, 
so that the left hand side is zero. 
Now we claim $\nabla_{\nabla r_p}\nabla r_p=0$
i.e., 
$\nabla r_p$ is a null vector for $\nabla_{\nabla r_p}$. 
For this, it suffices to show that for any smooth vector field $X$ on $M\setminus\{p\}$
\begin{align}
\langle \nabla_{\nabla r_p}\nabla r_p,X\rangle =0.\label{eq:Nablaradial}
\end{align}
This is true if $X$ is parallel to $\nabla r_p$, because 
for $f\in C^{\infty}(M\setminus\{p\})$ 
\begin{align*}
\langle \nabla_{\nabla_{r_p}}\!\nabla r_p,f\nabla r_p\rangle &=f
\langle \nabla_{\nabla_{r_p}}\!\nabla r_p,\nabla r_p\rangle =
f\frac12(\nabla r_p)|\nabla r_p|^2=0.
\end{align*}
Moreover, \eqref{eq:Nablaradial} holds if $X$ is vertical to 
$\nabla r_p$, because 
\begin{align*}
\langle \nabla_{\nabla r_p}\nabla r_p,X\rangle &=\frac12(\nabla r_p)\langle \nabla r_p,X\rangle =\frac12(\nabla r_p)0=0.
\end{align*}
Hence $\nabla_{\nabla r_p}$ has at most $n-1$ non-zero eigenvalues and by the Cauchy-Schwarz inequality, it holds on $({\rm Cut}(p)\cup \{p\})^c$ that  (see \cite{WylieYeroshkin})
\begin{align}
|{\rm Hess}\;r_p|^2=\|\nabla_{\nabla r_p}\|^2\geq \frac{\;(\Delta r_p)^2\;}{n-1}.\label{eq:HessTrace}
\end{align}
Now $m\leq 1$. Hence
\begin{align}
0&\geq \frac{\;(\Delta r_p)^2\;}{n-m}+{\Ric}_{m,n}(\Delta_V)
\left(\nabla r_p,\nabla r_p \right)-\frac{1}{\;n-m\;}|\langle V,\nabla r_p\rangle |^2+\langle \nabla \Delta_Vr_p,\nabla r_p\rangle .\label{eq:Bochner}
\end{align}
This gives us the following inequality  along $\gamma$, 
\begin{align}
\frac{\d}{\d r}(\Delta_Vr_p)(r,\theta)&\leq -\frac{\;(\Delta r_p(r,\theta))^2\;}{n-m}-{\Ric}_{m,n}(\Delta_V)
\left(\dot{\gamma}_r,\dot{\gamma}_r \right)+\frac{1}{\;n-m\;}|\langle V,\nabla r_p\rangle (r,\theta)|^2.\label{eq:equaitonalonggama}
\end{align}
From this, we have
\begin{align*}
\frac{\d\lambda}{\d s}&= C_p^{-1} e^{\frac{\;2V_{\gamma}(r)\;}{n-m}}\frac{\d\lambda}{\d r}\\
&= C_p^{-2} e^{\frac{\;2V_{\gamma}(r)\;}{n-m}}\left\{\left(\frac{\d}{\d r}e^{\frac{\;2V_{\gamma}(r)\;}{n-m}} \right)\Delta_Vr_p(r,\theta) + e^{\frac{\;2V_{\gamma}(r)\;}{n-m}} \frac{\d}{\d r} \Delta_Vr_p(r,\theta)\right\}\\
&=  C_p^{-2} e^{\frac{\;2V_{\gamma}(r)\;}{n-m}}\left\{e^{\frac{\;2V_{\gamma}(r)\;}{n-m}}\frac{2}{\;n-m\;}\cdot\frac{\;\partial V_{\gamma}(r)\;}{\partial r}\cdot \Delta_Vr_p(r,\theta)+e^{\frac{\;2V_{\gamma}(r)\;}{n-m}} \frac{\d}{\d r} \Delta_Vr_p(r,\theta)\right\}\\
&= C_p^{-2} e^{\frac{\;4V_{\gamma}(r)\;}{n-m}}\left\{ \frac{2}{\;n-m\;}\cdot\frac{\;\partial V_{\gamma}(r)\;}{\partial r}\cdot \Delta_Vr_p(r,\theta)+\frac{\d}{\d r} \Delta_Vr_p(r,\theta)\right\}
\end{align*}
\begin{align*}
&\leq \frac{ C_p^{-2} }{n-m}e^{\frac{\;4V_{\gamma}(r)}{n-m}\;}\left\{2\frac{\;\partial V_{\gamma}(r)\;}{\partial r}\Delta_Vr_p(r,\theta)-(\Delta r_p(r,\theta))^2 +|\langle V,\nabla r_p\rangle (r,\theta)|^2\right\}\\
&\hspace{5cm}- C_p^{-2} e^{\frac{\;4V_{\gamma}(r)\;}{n-m}}{\Ric}_{m,n}(\Delta_V)\left(\dot{\gamma}_r,\dot{\gamma}_r\right) \\
&= -\frac{ C_p^{-2} }{\;n-m\;}e^{\frac{\;4V_{\gamma}(r)\;}{n-m}}(\Delta_Vr_p(r,\theta))^2- C_p^{-2} e^{\frac{\;4V_{\gamma}(r)\;}{n-m}}{\Ric}_{m,n}(\Delta_V)\left(\dot{\gamma}_r,\dot{\gamma}_r\right)
\\
&= -\frac{1}{\;n-m\;}\left( C_p^{-1} e^{\frac{\;2V_{\gamma}(r)\;}{n-m}}\Delta_Vr_p(r,\theta)\right)^2-
 C_p^{-2} e^{\frac{\;4V_{\gamma}(r)\;}{n-m}}{\Ric}_{m,n}(\Delta_V)\left(\dot{\gamma}_r,\dot{\gamma}_r\right)\\
&=-\frac{\lambda^2}{\;n-m\;}-
 C_p^{-2} e^{\frac{\;4V_{\gamma}(r)\;}{n-m}}{\Ric}_{m,n}(\Delta_V)\left(\dot{\gamma}_r,\dot{\gamma}_r\right).
\end{align*}
Here we use  \eqref{eq:equaitonalonggama} at the inequality above and use 
$\Delta_Vr_p=\Delta r_p-\langle V, \nabla r_p\rangle $ in the next equality. 
If the equality holds for \eqref{eq:lambda/ds} at some $x=(r_0,\theta)\notin{\rm Cut}(p)\cup \{p\}$, then the equality for 
\eqref{eq:equaitonalonggama} equivalently the equality for \eqref{eq:Bochner} at $x\notin {\rm Cut}(p)\cup \{p\}$ holds, i.e.,   
\begin{align*}
0&=\frac{\;(\Delta r_p)^2\;}{n-m}+{\Ric}_{m,n}(\Delta_V)
\left(\nabla r_p,\nabla r_p \right)-\frac{1}{\;n-m\;}|\langle V,\nabla r_p\rangle |^2+\langle \nabla \Delta_V r_p,\nabla r_p\rangle \\
&\geq \frac{\;(\Delta r_p)^2\;}{n-1}+{\Ric}_{m,n}(\Delta_V)
\left(\nabla r_p,\nabla r_p \right)-\frac{1}{\;n-m\;}|\langle V,\nabla r_p\rangle |^2+\langle \nabla \Delta_Vr_p,\nabla r_p\rangle  
\end{align*}
holds at $x\notin {\rm Cut}(p)\cup \{p\}$. 
This and $m\leq1$ yield 
\begin{align*}
\frac{m-1}{\;(n-m)(n-1)\;}(\Delta r_p)^2(x)=0.
\end{align*}
Thus $m=1$ or $\Delta r_p(x)=0$. 
Since $M$ has an upper bound $\kappa_{\eps}>0$ of the sectional curvature on some $B_{\eps}(p)\subset {\rm Cut}(p)^c$, the usual Laplacian comparison theorem tells us that  $\Delta r_p(x)\geq (n-1)\sqrt{\kappa_{\eps}}\cot(\sqrt{\kappa_{\eps}}r_p(x))>0$ for $0<r_p(x)<\eps$. Therefore we obtain $m=1$, in particular, the equality for \eqref{eq:HessTrace} holds at $x$. This implies that 
$\nabla_{\nabla r_p}$ at $x$ has at most one non-zero eigenvalue of multiplicity $n-1$. 
\end{proof}
Let $\kappa$ be a continuous function on $[0,+\infty[$ with 
respect to the parameter $s$. 
Assuming the curvature bound ${\Ric}_{m,n}(\Delta_V)_x(\nabla r_p,\nabla r_p)\geq (n-m)\kappa(s_p(x)) {e^{-\frac{\;4\phi_V(x)\;}{n-m}}}  C_p^2  $ for  $s_p(x)<S$ with $x\notin {\rm Cut}(p)\cup\{p\}$, we see 
${\Ric}_{m,n}(\Delta_V)(\dot{\gamma}_r,\dot{\gamma}_r)\geq (n-m)\kappa(s) {e^{-\frac{\;4V_{\gamma}(r)\;}{n-m}}}  C_p^2  $ for $s=s(r,\theta)<S$ with $0<r<d(p,{\rm Cut}(p))$.
From $(\ref{eq:lambda/ds})$ we have the usual Riccati inequality
\begin{align}
-\frac{\d\lambda}{\d s}(s)\geq (n-m)\kappa(s)+\frac{\lambda(s)^2}{\;n-m\;}\quad\text{ for }\quad s\in]0,S[\label{eq:RiccattiIneq}
\end{align}
with the caveat that it is in terms of the parameter $s$ instead of $r$. 
This gives us the following comparison estimate. 

\begin{lemma}
\label{lem:LaplacianComparisonconformal}
Suppose that $(M,g)$ be an $n$-dimensional complete Riemannian manifold and $V$ a $C^1$-vector field. Fix $R\in]0,+\infty[$ and $x,p\in M$. 
Assume that \eqref{eq:RiciLowBdd} holds for $r_p(x)<R$ 
with $x\notin({\rm Cut}(p)\cup\{p\})$.   
Let $\gamma$, $s$, and $\lambda$ be defined to be as in Lemma~\ref{lem:RiccatiDiferIneq}. Then 
\begin{align}
\lambda(r,\theta)\leq m_{\kappa}(s) \label{eq:LocalLapComp}
\end{align}
holds for $r<R$, $s<\delta_{\kappa}$ and $x=(r,\theta)\notin\text{\rm Cut}(p)\cup\{p\}$. 
Here 
\begin{align*}
s=s_p(r)=  C_p \int_0^r\exp\left(-\frac{\;2\phi_V(\gamma_t)\;}{n-m} \right)\d t.
\end{align*} 
Suppose further that the equality in 
\eqref{eq:LocalLapComp} holds for some $r_0<R$ with $s_0:=s(r_0)<\delta_{\kappa}$. 
We choose an orthonormal basis $\{e_i\}_{i=1}^n$ of $T_pM$ 
with $e_n=\dot{\gamma}_0$. 
Let $\{Y_i\}_{i=1}^{n-1}$ be the Jacobi fields along $\gamma$ with $Y_i(0)=o_p$ and $Y_i'(0)=e_i$.
Then we have $m=1$, 
and at $x=(r,\theta)$ with $r\leq r_0$, $\nabla_{\nabla r_p}$ has at most one non-zero eigenvalue which is of multiplicity $n-1$, 
 and for all $r\in]0,r_0]$ we have 
\begin{align}
\Ric_{1,n}(\Delta_V)(\dot\gamma_r,\dot\gamma_r)&=(n-1)\kappa(s_p(\gamma_r))e^{-\frac{\;4V_{\gamma}(r)\;}{n-1}}  C_p^2  .\label{eq:Einstein}
\end{align} 
Moreover, for all $i$ we have 
$Y_i(r)=C_p^{-1}F_{\kappa}(r)E_i(r)$ for $r\in[0,r_0]$, where 
\begin{align}
F_{\kappa}(r):=\exp\left(\frac{\;V_{\gamma}(r)\;}{n-1} \right)\s_{\kappa}(s_p(\gamma_r)),\label{eq:parallel}
\end{align}
and $\{E_i(r)\}_{i=1}^{n-1}$ are the parallel vector fields with $E_i(0)=e_i$. 
Consequently,
\begin{align}
g_{\gamma_r}&=dr^2+  C_p^{-2} e^{\frac{\;2V_{\gamma}(r)\;}{n-1}}\s_{\kappa}^2(s_p(\gamma_r))g_{\mathbb{S}^{n-1}}.\label{eq:roundshere1}
\end{align}
Here $g_{\mathbb{S}^{n-1}}$ is the standard metric on the sphere $\mathbb{S}^{n-1}$.
\end{lemma}
\begin{proof}
Set $S:=s_p(R)$. Then $r<R$ implies $s<S$.
Since $\Delta r_p(r,\theta)\to+\infty$ as $r\to0$, we see 
$\lambda(r,\theta)\to+\infty$ as $r\to0$ or $s\to0$. 
We set $\beta(s):=\s_{\kappa}^2(s)(\lambda-m_{\kappa}(s))$. Then, 
by \eqref{eq:RiccatiEqM} and \eqref{eq:RiccattiIneq}, for $s<S$ 
\begin{align*}
\beta'(s)&=2\s_{\kappa}'(s)\s_{\kappa}(s)(\lambda-m_{\kappa}(s))+\s_{\kappa}^2(s)\left(\frac{\d\lambda}{\d s}-m_{\kappa}'(s)\right)\\
&=2\s_{\kappa}^2(s)\cot_{\kappa}(s)(\lambda-m_{\kappa}(s))
+\s_{\kappa}^2(s)\left(\frac{\d\lambda}{\d s}+(n-m)\kappa(s)+\frac{m_{\kappa}^2(s)}{\;n-m\;}
 \right)\\
 &\leq
 \frac{\s_{\kappa}^2(s)}{\;n-m\;}\left(2m_{\kappa}(s)\lambda-2m_{\kappa}^2(s)\right)
+\frac{\s_{\kappa}^2(s)}{\;n-m\;}\left(m_{\kappa}^2(s)-\lambda^2
 \right)\\
 &=-\frac{\s_{\kappa}^2(s)}{\;n-m\;}\left(\lambda-m_{\kappa}(s) \right)^2\leq0.
\end{align*}
We note here that \eqref{eq:RiccattiIneq} is derived from \eqref{eq:RiciLowBdd}. 
If we show $\beta(0)=0$, then $\beta(s)\leq \beta(0)=0$. For this, it suffices to prove that $s(\lambda-m_{\kappa}(s))$ is upper bounded as $s\to0$. We already know that $\lim_{s\to0}s\, m_{\kappa}(s)=n-m$ and 
the ratio $s/r=s_p(r)/r$ converges to $C_p$  as $r\to0$. So it suffices to prove  $\lim_{r\to0}r\lambda(r,\theta)=C_p^{-1}(n-1)$ as $r\to0$, equivalently $\lim_{r\to0}r\Delta r_p(r,\theta)=n-1$, because $\lim_{r\to0}r\langle V,\nabla r_p\rangle (r,\theta)=0$. In view of the usual Laplacian comparison theorem for the Laplace-Bertrami operator $\Delta$ under the upper (resp.~lower) bound $K_{\eps}$ (resp.~$\kappa_{\eps}$) of sectional curvature on $B_{\eps}(p)$, we see 
$(n-1)\cot_{K_{\eps}}(r)\leq\Delta r_p(r,\theta)\leq (n-1)\cot_{\kappa_{\eps}}(r)$ on $B_{\eps}(p)$. 
This implies the desired assertion.
Next we assume that the equality in \eqref{eq:LocalLapComp} holds for some $r_0<R$, i.e., $\lambda(r_0,\theta)=(n-m)\cot_{\kappa}(s_0)$ for $r_0<R$ with $s_0=s(r_0)$. 
This implies $0=\beta(s_0)\leq\beta(s)\leq\beta(0)=0$, hence 
$\lambda(r)=m_{\kappa}(s)$ for all $s\in[0,s_0]$. From this,   
\begin{align*}
\frac{\d\lambda}{\d s}(s_0)=\frac{\d m_{\kappa}}{\d s}(s_0).
\end{align*}
In particular, we have at $r_0$
\begin{align}
\frac{\d \lambda}{\d s}&\leq -\frac{\lambda(r)^2}{\;n-m\;}-
 C_p^{-2} e^{\frac{\;4V_{\gamma}(r)\;}{n-m}}
\Ric_{m,n}(\Delta_V)(\dot\gamma_r,\dot\gamma_r)\nonumber\\
&\leq 
-\frac{\lambda(r)^2}{\;n-m\;}-
(n-m)\kappa(s)=-\frac{m_{\kappa}(s)^2}{\;n-m\;}-
(n-m)\kappa(s)=\frac{\d \lambda}{\d s}.\label{eq:RiccatiIneq}
\end{align}
Then the equality holds in  \eqref{eq:lambda/ds} 
at $x=(r_0,\theta)$. So we have $m=1$ by Lemma~\ref{lem:RiccatiDiferIneq}. 
We can conclude $\beta(s)\equiv0$ on $[0,s_0]$ from $\beta(0)=\beta(s_0)=0$ and $\beta'(s)\leq0$ so that $\lambda(r,\theta)=(n-1)\cot_{\kappa}(s)$ for $s\in]0,s_0]$. 
We then see the equality \eqref{eq:RiccatiIneq} at any $r\in]0,r_0]$, hence \eqref{eq:Einstein} holds at any $r\in]0,r_0]$. 

Finally we prove \eqref{eq:roundshere1} at any $r\in]0,r_0]$ under $\lambda(r_0)=(n-m)\cot_{\kappa}(s_0)$. 
Hereafter, we assume $r\in]0,r_0]$. 
By Lemma~\ref{lem:RiccatiDiferIneq}, at $x=(r,\theta)$, $\nabla_{\nabla r_p}$ has a non-zero eigenvalue $A(r)$ which is of $n-1$ multiplicity. 
Then we have 
\begin{align*}
\lambda(r,\theta)&= C_p^{-1} e^{\frac{\;2V_{\gamma}(r)\;}{n-1}}(\Delta r_p(r,\theta)-\langle V,\nabla r_p\rangle (r,\theta))\\
&= C_p^{-1} e^{\frac{\;2V_{\gamma}(r)\;}{n-1}}((n-1)A(r)-\langle V,\nabla r_p\rangle (r,\theta))=(n-1)\cot_{\kappa}(s),
\end{align*}
 where we use the equality \eqref{eq:LocalLapComp} at any $r\in]0,r_0]$. 
So we have $A(r)= C_p e^{-\frac{\;2V_{\gamma}(r)\;}{n-1}}\cot_{\kappa}(s)+\frac{\;\langle V,\nabla r_p\rangle (r,\theta)\;}{n-1}
 =(n-1)^{-1}\Delta r_p(\gamma_r)$. 
The radial curvature equation (see \cite[Theorem~2 in pp.~44]{Pet:RiemannianGeo}) tells us that 
\begin{align}
R(E_i,\dot{\gamma}_r)\dot{\gamma}_r=-(A'(r)+A(r)^2)E_i.\label{eq:radialCurvatureEq}
\end{align}
Combining Bochner-Weitzenb\"ock formula with \eqref{eq:Einstein}, 
we have 
\begin{align}
A'(r)+A(r)^2=\frac{\;V_{\gamma}''(r)\;}{n-1}+\left(\frac{\;V_{\gamma}'(r)\;}{n-1} \right)^2-\kappa(s_p(\gamma_r))e^{-\frac{\;4V_{\gamma}(r)\;}{n-1}}C_p^2=\frac{\;F_{\kappa}''(r)\;}{F_{\kappa}(r)}.\label{eq:Ar}
\end{align}
Since $F_{\kappa}(0)=0$ and $F_{\kappa}'(0)=C_p$, we obtain 
$$
Y_i(r)=C_p^{-1}F_{\kappa}(r)E_i(r)=C_p^{-1}e^{\frac{\;V_{\gamma}(r)\;}{n-1}}\s_{\kappa}(s_p(\gamma_r))E_i(r).
$$ 
This proves the desired conclusion. 
\end{proof}

\begin{corollary}\label{cor:Cutlocus}
Let $(M,g)$ be an $n$-dimensional complete Riemannian manifold and $V$ a $C^1$-vector field. Fix $p\in M$ and $R\in]0,+\infty[$. 
Assume that \eqref{eq:RiciLowBdd} holds for $r_p(x)<R$ 
with $x\notin {\rm Cut}(p)\cup\{p\}$.  Then $s_p(x)<\delta_{\kappa}$. 
\end{corollary}
\begin{proof}
We may assume $\delta_{\kappa}<\infty$.   
Take $x\in B_R(p)$  with 
$x\notin{\rm Cut}(p)\cup\{p\}$.  
Let $x=(r,\theta)$ be the polar coordinate expression around $p$ and set 
$s:=s_p(r)= C_p \int_0^{r}\exp\left(-\frac{2 \phi_V (\gamma_t)}{n-m} \right)\d t$ and 
$S=s_p(R)$, where $\gamma$ is a unit speed geodesic with 
$\gamma_0=p$ and $\dot{\gamma}_0=\theta$. 
We see $s_p(x)<S$. 
Assume $S>\delta_{\kappa}$. 
Then there exists $r_0\in]0,R[$ such that $\delta_{\kappa}= C_p \int_0^{r_0}\exp\left(-\frac{2V_{\gamma}(t)}{n-m} \right)\d t$. By \eqref{eq:LocalLapComp},   
$\lambda(r,\theta)\leq (n-m)\cot_{\kappa}(s)$ holds for $s<\delta_{\kappa}$. 
Since $r\uparrow r_0$ is equivalent to $s=s(r)\uparrow\delta_{\kappa}$, we have  
$$
\lambda(r_0,\theta)=\lim_{r\uparrow r_0}\lambda(r,\theta)\leq 
\lim_{r\uparrow r_0}(n-m)\cot_{\kappa}(s(r))=-\infty.
$$ This contradicts the 
well-definedness of $\lambda(r,\theta)= C_p^{-1} \left(e^{\frac{2\phi_V}{n-m}}\Delta_Vr_p\right)(r,\theta)$ for $r\in]0,R[$. 
Therefore $S\leq\delta_{\kappa}$ under $\delta_{\kappa}<\infty$ and we obtain the conclusion $s_p(x)<S\leq\delta_{\kappa}$. 
\end{proof}

Let $p\in M$ and let $(r,\theta)$, $r>0$, $\theta\in \mathbb{S}^{n-1}$ be exponential 
polar coordinates (for the metric $g$) around $p$ which are defined on 
a maximal  star shaped domain in $T_pM$ called the 
\emph{segment domain}. Write the volume element $\d\m=J(r,\theta)\d r\land\d\theta$. 

Let $s_p(\cdot)$ be the re-parametrized distance function defined above. Inside the segment domain, 
$s_p$ has the simple formula
\begin{align*}
s_p(r,\theta)= C_p \int_0^re^{-\frac{\;2\phi_V(t,\theta)\;}{n-m}}\d t.
\end{align*}
Therefore, $s_p$ is a smooth function in the segment domain with the property that 
$\frac{\partial s}{\partial r}= C_p e^{-\frac{\;2\phi_V(r,\theta)\;}{n-m}}$. We can then also take $(s,\theta)$ 
to be coordinates which are also valid for the entire segment theorem. We can not control the derivative of $s$ in directions tangent to the sphere, so the new 
$(s,\theta)$ coordinates are \emph{not} orthogonal as in the case for geodesic polar coordinates. However, this is not the issue when we computing volumes as 
\begin{align}
\left.\begin{array}{rl}e^{-\frac{\;2\phi_V}{n-m}\;}\d\mu_V&=e^{-\frac{n-m+2}{n-m}\phi_V}J(r,\theta)\d r\land\d\theta\\
&=C_p^{-1}e^{-\phi_V}J(r,\theta)\d s\land\d\theta.\end{array}\right.\label{eq:conformal}
\end{align} 
Here $\d\mu_V=e^{-\phi_V}\d \m$. 
We denote the derivative in the
radial direction in terms of this parameter by $\frac{\d}{\d s}$. 
In geodesic polar coordinates $\frac{\d}{\d s}$ has the expression $\frac{\d }{\d s}= C_p^{-1} e^{\frac{\;2\phi_V(r,\theta)\;}{n-m}}\frac{\partial}{\partial r}$. 
Note that it is not the same as $\frac{\partial}{\partial s}$ in $(s,\theta)$ coordinates. 

\begin{proof}[Proof of Theorem~\ref{thm:GlobalLapComp}]
The implication \eqref{eq:RiciLowBdd}$\Longrightarrow$\eqref{eq:GloLapComp} 
for $R<\infty$ 
follows from Lemma~\ref{lem:LaplacianComparisonconformal}, because 
$r_p$ is smooth on $M\setminus(\text{\rm Cut}(p)\cup\{p\})$. 
The implication \eqref{eq:RiciLowBdd}$\Longrightarrow$\eqref{eq:GloLapComp} 
for $R=+\infty$ follows from it. 
\end{proof}

\section{Proofs of Theorem~\ref{thm:WeightedMyers} and Corollary~\ref{cor:WeightedMyers}}

\begin{proof}[Proof Theorem~\ref{thm:WeightedMyers}]
Suppose that there exist points $p,q\in M$ such that $s(p,q)>\delta_{\kappa}$. 
Since $\text{\rm Cut}(p)$ is closed and measure zero, we may assume $q\notin \text{\rm Cut}(p)$. By Lemma~\ref{lem:LaplacianComparisonconformal}, 
along minimal geodesic from $p$ to $q$, 
$\lambda(r,\theta)\leq m_{\kappa}(s)$. However, as 
$s\to\delta_{\kappa}$, $m_{\kappa}(s)\to-\infty$. This implies 
$\Delta r_p(x)\to-\infty$ as $s(p,x)\to\delta_{\kappa}$. This contradicts 
that $r_p$ is smooth in a neighborhood of $q$. 
The final assertion follows Remark~\ref{rem:SpRp}. 
\end{proof}
\begin{proof}[Proof of Corollary~\ref{cor:WeightedMyers}]
Suppose that $\sup_{q\in M}d(p,q)=+\infty$. Then there exists a sequence $\{q_i\}$ in $M$ such that $d(p,q_i)\to+\infty$ as $i\to+\infty$. 
By Lemma~\ref{lem:phimcomplete}, $s(p,q_{i})\to+\infty$ as $k\to+\infty$, 
which contradicts $\sup_{q\in M}s(p,q)\leq\delta_{\kappa}$. 
Therefore, $\sup_{q\in M}d(p,q)<\infty$, hence $M$ is compact. 
\end{proof}

\section{Proof of Theorem~\ref{thm:BGVol}}
Recall that for a Riemannian manifold $\frac{\d}{\d r}\log J(r,\theta)=\Delta r_p(r,\theta)$, where 
$\Delta r_p$ is the standard Laplacian acting on the distance function $r_p$ from the point $p$.  \eqref{eq:conformal} indicates we should consider the quantity 
\begin{align}
\frac{\d}{\d s}\log(e^{-V_{\gamma}(r)}J(r,\theta))= C_p^{-1}e^{\frac{\;2V_{\gamma}(r)\;}{n-m}}\left(\Delta r_p(r,\theta)-\langle V_{\gamma_r},\dot\gamma_r\rangle  \right)= C_p^{-1} 
e^{\frac{\;2V_{\gamma}(r)\;}{n-m}}\Delta_V r_p(r,\theta)
.\label{eq:quantity}
\end{align}

\begin{lemma}[Volume Element Comparison]\label{lem:VolComp}
Let $(M,g)$ be an $n$-dimensional complete Riemannian manifold and $V$ a $C^1$-vector field. Fix $p\in M$ and $R\in]0,+\infty]$.  
Assume that \eqref{eq:RiciLowBdd} holds for $r_p(x)<R$ with $x\notin {\rm Cut}(p)\cup\{p\}$. 
Let $J$ be the volume element in geodesic polar coordinates around $p\in M$ and set $J_V(r,\theta):=e^{-V_{\gamma}(r)}J(r,\theta)$. Then for $r_0<r_1<R$ with 
$r_1<{\rm cut}(\theta)$, 
\begin{align}
\frac{\;J_V(r_1,\theta)\;}{J_V(r_0,\theta)}
\leq \frac{\;\s_{\kappa}(s_p(r_1,\theta))^{n-m}\;}{\s_{\kappa}(s_p(r_0,\theta))^{n-m}}.\label{eq:VolComp} 
\end{align}
Here $\text{\rm cut}(\theta)$ is the distance from $p$ to the cut point along the geodesic with $\gamma(0)=p$ and $\dot{\gamma}(0)=\theta$. 
\end{lemma}
\begin{proof}
Recall $s=s_p(r)=s_p(r,\theta)= C_p \int_0^r\exp\left(-\frac{\;2V_{\gamma}(t)\;}{n-m} \right)\d t$ and 
$\gamma$ is the unit speed geodesic from $p$ with $\dot{\gamma}_0=\theta$.
First note that the right hand side of \eqref{eq:VolComp} is meaningful for 
$r_0<r_1<R$. Indeed, if $R<+\infty$,  $s_p(r_0,\theta)<s_p(r_1,\theta)<\delta_{\kappa}$ 
by Corollary~\ref{cor:Cutlocus}. If $R=+\infty$, we can take $R_0\in]r_1,+\infty[$ so that \eqref{eq:RiciLowBdd} holds for $r_p(x)<R_0$, hence $s_p(r_0,\theta)<s_p(r_1,\theta)<\delta_{\kappa}$ by Corollary~\ref{cor:Cutlocus}. 
From Lemma~\ref{lem:LaplacianComparisonconformal} and \eqref{eq:quantity} 
we have that 
\begin{align}
\frac{\d}{\d s}\log J_V(r,\theta)= C_p^{-1} e^{\frac{\;2 \phi_V\;}{n-m}} \Delta_V r_p(r,\theta)\leq (n-m)\cot_{\kappa}(s)
=\frac{\d}{\d s}\log(\s_{\kappa}(s)^{n-m})\label{eq:JacobIneq}
\end{align}
for $r\in]0,R\land {\rm cut}(\theta)[$.   
Integrating \eqref{eq:JacobIneq} between any $s_0< s_1
$  
with $s_i=s_p(r_i,\theta)$ and $r_i\in]0,R\land {\rm cut}(\theta)[$ $(i=0,1)$ gives 
\begin{align*}
\log\left(
\frac{\;J_V(r_1,\theta)\;}{J_V(r_0,\theta)} \right)
\leq \log\left(\frac{\;\s_{\kappa}(s_1)^{n-m}\;}{\s_{\kappa}(s_0)^{n-m}} \right)\quad\text{ implies }
\quad \frac{\;J_V(r_1,\theta)\;}{J_V(r_0,\theta)}
\leq \frac{\;\s_{\kappa}(s_1)^{n-m}\;}{\s_{\kappa}(s_0)^{n-m}}
\end{align*}
for all $r_0< r_1<R\land {\rm cut}(\theta)$. Note that since $\d s$ is an orientation preserving change of variables along the geodesic $\gamma$, the quantity is also non-increasing in terms of the parameter $r\in]0,R\land {\rm cut}(\theta)[$. 
\end{proof}

\begin{proof}[Proof of Theorem~\ref{thm:BGVol}]
By Lemma~\ref{lem:VolComp}, for all $r_1,r_2>0$ with $r_1< r_2<R$ and 
$r_2<{\rm cut}(\theta)$
\begin{align*}
\frac{\;J_V(r_2,\theta)\;}{J_V(r_1,\theta)}\leq \frac{\;\s_{\kappa}^{n-m}(s_p(r_2,\theta))\;}{\s_{\kappa}^{n-m}(s_p(r_1,\theta))}\leq 
\frac{\;\s_{\kappa}^{n-m}\left( 
\sup_{\eta\in\mathbb{S}^{n-1}}s_p(r_2,\eta)
\right)\;}{\s_{\kappa}^{n-m}\left(\inf_{\eta\in\mathbb{S}^{n-1}}s_p(r_1,\eta) 
\right)}.
\end{align*}
So for 
$0\leq r_a< r_b\leq r_d$, $0\leq r_a\leq r_c< r_d$ and  $r_d<R$, we have following inequality
\begin{align*}
\frac{\;\int_{\text{\rm cut}(\theta)\land r_c}^{\text{\rm cut}(\theta)\land r_d}J_V(r_2,\theta)\d r_2\;}{\int_{\text{\rm cut}(\theta)\land r_a}^{\text{\rm cut}(\theta)\land r_b}J_V(r_1,\theta)\d r_1}&\leq\frac{\;\int_{\text{\rm cut}(\theta)\land r_c}^{\text{\rm cut}(\theta)\land r_d}\s_{\kappa}^{n-m}(s_p(r_2,\theta))\d r_2\;}{\int_{\text{\rm cut}(\theta)\land r_a}^{\text{\rm cut}(\theta)\land r_b}\s_{\kappa}^{n-m}(s_p(r_1,\theta))\d r_1}
\\
&\leq 
\frac{\;\int_{r_c}^{r_d}\s_{\kappa}^{n-m}\left(\sup_{\eta\in\mathbb{S}^{n-1}}s_p(r_2,\eta)\right)\d r_2\;}{\int_{r_a}^{r_b}\s_{\kappa}^{n-m}\left(\inf_{\eta\in\mathbb{S}^{n-1}}s_p(r_1,\eta) 
\right)\d r_1}
\end{align*}
under $r_a=r_c$ or $r_b=r_d$ by use of \cite[Lemma~3.1]{Zhu97} 
(cf.~\cite[Proof of Theorem~3.2]{Zhu97}). 
From this, we can deduce that 
\begin{align*}
\frac{\;\int_{\mathbb{S}^{n-1}}\int_{\text{\rm cut}(\theta)\land r_c}^{\text{\rm cut}(\theta)\land r_d}J_V(r_2,\theta)\d r_2\d\theta\;}{\int_{\mathbb{S}^{n-1}}\int_{\text{\rm cut}(\theta)\land r_a}^{\text{\rm cut}(\theta)\land r_b}J_V(r_1,\theta)\d r_1\d\theta}
\leq
\frac{\;\int_{\mathbb{S}^{n-1}}\int_{r_c}^{r_d}\s_{\kappa}^{n-m}\left(\sup_{\eta\in\mathbb{S}^{n-1}}s_p(r_2,\eta) 
\right)\d r_2\d\theta\;}{\int_{\mathbb{S}^{n-1}}\int_{r_a}^{r_b}\s_{\kappa}^{n-m}\left(\inf_{\eta\in\mathbb{S}^{n-1}}s_p(r_1,\eta) 
\right)\d r_1\d\theta}
\end{align*}
holds for general $0\leq r_a< r_b\leq r_d$, $0\leq r_a\leq r_c< r_d$ and $r_d<R$.  
This implies that \eqref{eq:BGAnnuliUpLow} holds for $r_1<R$. 
If $\phi$ is rotationally symmetric around $p$, $s_p(r,\theta)$ can be written as $s_p(r)$ and one can derive 
\begin{align*}
\frac{\;\int_{\mathbb{S}^{n-1}}\int_{\text{\rm cut}(\theta)\land r_c}^{\text{\rm cut}(\theta)\land r_d}J_V(r_2,\theta)\d r_2\d\theta\;}{\int_{\mathbb{S}^{n-1}}\int_{\text{\rm cut}(\theta)\land r_a}^{\text{\rm cut}(\theta)\land r_b}J_V(r_1,\theta)\d r_1\d\theta}\leq
\frac{\;\int_{\mathbb{S}^{n-1}}\int_{r_c}^{r_d}\s_{\kappa}^{n-m}(s_p(r_2))\d r_2\d\theta\;}{\int_{\mathbb{S}^{n-1}}\int_{r_a}^{r_b}\s_{\kappa}^{n-m}(s_p(r_1))\d r_1\d\theta}.
\end{align*} 
This implies that \eqref{eq:BGAnnuli} holds for $r_1<R$.   
Similarly, in the modified coordinates $(s,\theta)$, we set 
\begin{align*}
\text{\rm cut}_s(\theta):=\int_0^{\text{\rm cut}(\theta)}
e^{-\frac{\;2V_{\gamma}(t)\;}{n-m}}\d t,
\end{align*} 
where $\gamma$ is the unit speed geodesic with $\gamma_0=p$ and $\dot{\gamma}_0=\theta$. Then we have 
\begin{align*}
\nu_V(C(p,s_0,s_1))=\int_{\mathbb{S}^{n-1}}\int_{\text{\rm cut}_s(\theta)\land s_0}^{\text{\rm cut}_s(\theta)\land s_1}J_V(r(s,\theta),\theta)\d s\d \theta,
\end{align*}
and 
\begin{align*}
v(\kappa,s_0,s_1)=\int_{\mathbb{S}^{n-1}}\int_{s_0}^{s_1}\s_{\kappa}^{n-m}(s)\d s\d \theta=\omega_{n-1}\int_{s_0}^{s_1}\s_{\kappa}^{n-m}(s)\d s.
\end{align*}
Therefore, \eqref{item:BG2} follows.
Here $r(s,\theta):= C_p^{-1} \int_0^s\exp\left(\frac{\;2V_{\gamma}(f^{-1}(u))\;}{n-m} \right)\d u$ with 
$f(r):=s_p(r,\theta)$.  
Note that $s_1<\delta_{\kappa}$ always holds under the condition. 
Indeed, $s_1<S$ implies $s_1<\delta_{\kappa}$ under $R<+\infty$ by 
Corollary~\ref{cor:Cutlocus}. When $R=+\infty$, for any $\theta\in\mathbb{S}^{n-1}$ there exists $R_0\in]0,+\infty[$ 
depending on $\theta$ such that $s_1<s(R_0,\theta)$. Then applying Corollary~\ref{cor:Cutlocus} for $R_0<\infty$,
$r_1:=r(s_1,\theta)<r(s(R_0,\theta),\theta)=R_0$ implies 
$s_1=s(r(s_1,\theta),\theta)<\delta_{\kappa}$, where we use \eqref{eq:RiciLowBdd} holds for 
$r_p(x)<R_0$. 
\end{proof}

\begin{proof}[Proof of Corollary~\ref{cor:BGVol}]
By Theorem~\ref{thm:BGVol}\eqref{item:BG1}, for 
$0<r_1<r_2<R$ 
\begin{align*}
\frac{\;\mu_V(B_{r_2}(p))\;}{\mu_V(B_{r_1}(p))}&\leq \frac{\;\int_0^{r_2}\left( C_p e^{-\frac{\;2\underline{\phi}_V(r)\;}{n-m}}r\right)^{n-m}\d r\;}{\int_0^{r_1}\left(  C_p e^{-\frac{\;2\overline{\phi}_V(r)\;}{n-m}}r\right)^{n-m}\d r}\\
&\leq 
e^{2(\overline{\phi}_V(r_1)-\underline{\phi}_V(r_2))}\frac{\;\int_0^{r_2} r^{n-m}\d r\;}{\int_0^{r_1} r^{n-m}\d r}
=e^{2(\overline{\phi}_V(r_1)-\underline{\phi}_V(r_2))}\left(\frac{r_2}{r_1} \right)^{n-m+1}.
\end{align*}
\end{proof}

\section{Proofs of Theorem~\ref{thm:AmbroseMyers}, Corollaries~\ref{cor:AmbroseMyersVbdd} and \ref{cor:AmbroseMyers}}

\begin{proof}[Proof of Theorem~\ref{thm:AmbroseMyers}]
Suppose that $M$ is non-compact. Then there exists a unit speed geodesic $\gamma$ with $\gamma_0=p$ satisfying \eqref{eq:Ambrose}. Note that the function $\lambda(t)$ is smooth for all $t>0$ along $\gamma$. 
By \eqref{eq:lambda/dr}, we have 
\begin{align*}
\lambda(t)-\lambda(1)+\frac{C_p}{n-m}\int_1^t e^{-\frac{2V_{\gamma}(r)}{n-m}}\lambda(r)^2\d r
&\leq - C_p^{-1} 
\int_1^t 
e^{\frac{\;2V_{\gamma}(r)\;}{n-m}}\text{
\rm Ric}_{m,n}(\Delta_V)(\dot{\gamma}_r,\dot{\gamma}_r)\d r.
\end{align*} 
Hence 
\begin{align}
\lim_{t\to+\infty}\left(\lambda(t)+\frac{C_p}{n-m}\int_1^t e^{-\frac{\;2V_{\gamma}(r)\;}{n-m}}\lambda(r)^2\d r\right)=-\infty.\label{eq:MyerEstimate}
\end{align}
In particular, $\lim_{t\to+\infty}\lambda(t)=-\infty$. 
Next we prove that there exists a finite number $T > 0$ such
that $\lim_{t\to T-}\lambda(t)=-\infty$, which contradicts the smoothness of $\lambda(r)$. By \eqref{eq:MyerEstimate}, 
given $C>n-m$ there exists $t_0>1$ such that 
$$
-\lambda(t_0)-\frac{C_p}{n-m}\int_1^{t_0}e^{-\frac{\;2V_{\gamma}(r)\;}{n-m}}\lambda(r)^2\d r\geq \frac{C}{n-m}.
$$
Since 
$$
\lim_{t\to+\infty}\int_1^t
e^{\frac{\;2V_{\gamma}(r)\;}{n-m}} \text{
\rm Ric}_{m,n}(\Delta_V)(\dot{\gamma}_r,\dot{\gamma}_r)\d r=+\infty,
$$ 
there exists $t_1\in]t_0,+\infty[$ such that 
$\int_{t_0}^t
e^{\frac{\;2V_{\gamma}(r)\;}{n-m}} \text{
\rm Ric}_{m,n}(\Delta_V)(\dot{\gamma}_r,\dot{\gamma}_r)\d r\geq0$ for all $t\geq t_1$. 
Let $\psi(t)$ be the function defined by 
\begin{align}
\psi(t):=-\lambda(t)-\frac{C_p}{n-m}\int_1^te^{-\frac{\;2V_{\gamma}(r)}{n-m}\;}\lambda(r)^2\d r- C_p^{-1} \int_1^t e^{\frac{\;2V_{\gamma}(r)\;}{n-m}} \text{
\rm Ric}_{m,n}(\Delta_V)(\dot{\gamma}_r,\dot{\gamma}_r)\d r.\label{eq:MyerEstimateInt}
\end{align}
Then we see $\psi'(t)\geq0$ by \eqref{eq:lambda/dr}. 
Hence $\psi(t)\geq\psi(t_0)$ for $t\geq t_1>t_0$. This implies that 
\begin{align}
-\lambda(t)-\frac{C_p}{n-m}\int_1^t e^{-\frac{\;2V_{\gamma}(r)}{n-m}\;}\lambda(r)^2\d r\geq\frac{C}{n-m}>1\label{eq:Stpe1}
\end{align}
holds for all $t\geq t_1$. 
Let us consider the sequence $\{t_{\ell}\}$ defined inductively by 
\begin{align*}
C_p\int^{t_{\ell+1}}_{t_{\ell}}e^{-\frac{\;2V_{\gamma}(r)\;}{n-m}}
\d r
=(n-m)\left(\frac{n-m}{C} \right)^{\ell-1}\quad \text{ for }\quad\ell\geq1.
\end{align*}
The existence of such sequence is guaranteed by the 
$(V,m)$-completeness of $(M,g,V)$ at $p$.     
Let $T$ be the increasing limit of $\{t_{\ell}\}$. Then we see 
\begin{align*}
C_p\int_{t_1}^Te^{-\frac{\;2V_{\gamma}(r)\;}{n-m}}
\d r=\frac{C(n-m)}{C-n+m}. 
\end{align*}
In view of the $(V,m)$-completeness of $(M,g,V)$ at $p$, we have 
$$
\int_1^{\infty}e^{-\frac{\;2V_{\gamma}(r)\;}{n-m}}
\d r=+\infty.
$$
Thus we obtain $T<\infty$. 
Finally  we claim that for given $\ell\in\mathbb{N}$, $-\lambda(t)\geq\left( \frac{C}{n-m}\right)^{\ell}$ for all $t\geq t_{\ell}$. This is true for $\ell=1$ by \eqref{eq:Stpe1}. Suppose that $-\lambda(r)\geq\left( \frac{C}{n-m}\right)^{\ell}$ for all $r\geq t_{\ell}$ and fix $t\geq t_{\ell+1}$. Then using inequality \eqref{eq:Stpe1} again, 
\begin{align*}
-\lambda(t)&\geq\frac{C}{n-m}+\frac{C_p}{n-m}\int_1^{t_{\ell}}
e^{-\frac{\;2V_{\gamma}(r)\;}{n-m}}\lambda(r)^2\d r+ 
\frac{C_p}{n-m}\int_{t_{\ell}}^{t_{\ell+1}}
e^{-\frac{\;2V_{\gamma}(r)\;}{n-m}}\lambda(r)^2\d r\\
&\geq  
\frac{C_p}{n-m}\int_{t_{\ell}}^{t_{\ell+1}}
e^{-\frac{\;2V_{\gamma}(r)\;}{n-m}}\lambda(r)^2\d r\\
&\geq 
 \frac{C^{2\ell}}{(n-m)^{2\ell}}\cdot\frac{(n-m)^{\ell-1}}{C^{\ell-1}}=\left(\frac{C}{n-m} \right)^{\ell+1}.
\end{align*}
Therefore we prove the claim. 
In particular, $\lim_{t\to T-}\lambda(t)=-\infty$ which is the desired contradiction. 
\end{proof}

\begin{proof}[Proof of Corollary~\ref{cor:AmbroseMyersVbdd}]
Suppose that there exists a non-negative integrable function $f$ on $[0,+\infty[$ satisfying $\langle V,\nabla r_p\rangle \geq -f(r_p)$. Then 
$V_{\gamma}(r)\geq -\int_0^rf(s)\d s\geq -\int_0^{\infty}f(s)\d s>-\infty$ and ${\Ric}_{m,n}(\Delta_V)\geq0$ imply 
\begin{align*}
\int_0^{\infty}&e^{\frac{\;2V_{\gamma}(t)}{n-m}\;}{\Ric}_{m,n}(\Delta_V)(\dot{\gamma}_t,\dot{\gamma}_t)\d t\\
&\geq \exp\left(-\frac{2}{n-m}\int_0^{\infty}f(s)\d s\right)\int_0^{\infty}{\Ric}_{m,n}(\Delta_V)(\dot{\gamma}_t,\dot{\gamma}_t)\d t=+\infty.
\end{align*}
This yields the conclusion by Theorem~\ref{thm:AmbroseMyers}.
\end{proof}

\begin{proof}[Proof of Corollary~\ref{cor:AmbroseMyers}]
Suppose that \eqref{eq:AmbroseRiccibounds} holds for 
every unit speed geodesic $\gamma$ emanating from $p$. 
The $(V,m)$-completeness of $(M,g,V)$ at $p$ implies 
\begin{align*}
\int_0^{\infty}e^{-\frac{\;2V_{\gamma}(t)\;}{n-m}}\d t=+\infty.
\end{align*}
Then we have
\begin{align*}
\int_0^{\infty}e^{\frac{\;2V_{\gamma}(t)\;}{n-m}}{\Ric}_{m,n}(\Delta_V)(\dot{\gamma}_t,\dot{\gamma}_t)\d t\geq (n-m)\kappa\, C_p^2 \int_0^{\infty}
e^{-\frac{\;2V_{\gamma}(t)\;}{n-m}}\d t=+\infty.
\end{align*}
This yields the conclusion by Theorem~\ref{thm:AmbroseMyers}.
\end{proof}

\section{Proof of Theorem~\ref{thm:ChengDiamSphere}}

For the proof of Theorem~\ref{thm:ChengDiamSphere}, we need 
the following lemma on the solution of Jacobi equation. 

\begin{lemma}\label{lem:SymJacob}
Let $\kappa:[0,\infty[\to\R$ be a continuous function and 
 $\s_{\kappa}$ the unique solution of the Jacobi equation $\s_{\kappa}''(s)+\kappa(s)\s_{\kappa}(s)=0$ with 
$\s_{\kappa}(0)=0$ and $\s_{\kappa}'(0)=1$, and $\delta_{\kappa}:=\inf\{s>0\mid \s_{\kappa}(s)=0\}$ the first zero point of $\s_{\kappa}$. Assume that  $\delta_{\kappa}<\infty$ and $\kappa(s)=\kappa(\delta_{\kappa}-s)$ holds for all $s\in[0,\delta_{\kappa}]$. Then 
$\s_{\kappa}'(\delta_{\kappa})=-1$, $\s_{\kappa}'(\delta_{\kappa}/2)=0$ and  
$\s_{\kappa}(s)=\s_{\kappa}(\delta_{\kappa}-s)$ for all $s\in[0,\delta_{\kappa}]$. 
\end{lemma}
\begin{proof}
Set $\overline{\s}_{\kappa}(s):=\s_{\kappa}(\delta_{\kappa}-s)$ for $s\in[0,\delta_{\kappa}]$. Then this satisfies 
$\overline{\s}_{\kappa}''(s)+\kappa(s)\overline{\s}_{\kappa}(s)=0$ and $\overline{\s}_{\kappa}(0)=0$ and $\overline{\s}_{\kappa}'(0)=-\s_{\kappa}'(\delta_{\kappa})$. If we prove 
$\overline{\s}_{\kappa}'(0)=1$, i.e., $\s_{\kappa}'(\delta_{\kappa})=-1$, then the uniqueness of the solution implies the assertion.  

Note that $s_{\kappa}(s):=\overline{\s}_{\kappa}(s)/ \overline{\s}_{\kappa}'(0)=-\s_{\kappa}(\delta_{\kappa}-s)/\s_{\kappa}'(\delta_{\kappa})$ also satisfies the Jacobi equation with $s_{\kappa}(0)=0$ and $s_{\kappa}'(0)=1$. 
Then the uniqueness implies $s_{\kappa}(s)=\s_{\kappa}(s)$, that is, $\s_{\kappa}(\delta_{\kappa}-s)=-\s_{\kappa}'(\delta_{\kappa})s_{\kappa}(s)$ for $s\in [0,\delta_{\kappa}]$, in particular, $\s_{\kappa}(\delta_{\kappa}/2)=-\s_{\kappa}'(\delta_{\kappa})\s_{\kappa}(\delta_{\kappa}/2)$. Therefore, $\s_{\kappa}'(\delta_{\kappa})=-1$ by $\s_{\kappa}(\delta_{\kappa}/2)>0$.  
The proof of $\s_{\kappa}'(\delta_{\kappa}/2)=0$ is easy from 
$\s_{\kappa}'(s)=-\s_{\kappa}'(\delta_{\kappa}-s)$ for $s\in[0,\delta_{\kappa}]$.   
\end{proof}

Hereafter, we assume $V=\nabla\phi$ for some $\phi\in C^2(M)$ and set $C_p:=\exp\left(-\frac{\;2\phi(p)\;}{n-m}\right)$ for the definition of $s_p(x)$ with $p$ being an arbitrary point.    
We now consider the conformal metric $h=e^{-\frac{\;4\phi\;}{n-m}}g$.  

\begin{lemma}\label{lem:Vindepenedent}
Fix $p\in M$. Suppose that there exists a point $q\in M$  such that 
$s(p,q)=d^{\,h}(p,q)$ and let $\gamma$ be the minimal unit speed $g$-geodesic from $p$ and $q$ such that $s(p,q)=\int_0^{d(p,q)}e^{-2\frac{\; \phi(\gamma_t)\;}{n-m}}\d t$. Then  $\nabla\phi$ is parallel to $\dot{\gamma}$ $($not parallel along $\gamma$$)$. 
Moreover if 
$s(p, x) = d^{\,h}(p, x)$ holds for any $x\in M$, then
 $\phi$ is rotationally symmetric around $p$. 
\end{lemma}
\begin{proof}
Since $t<d(p,q)$ implies $\gamma_t\notin {\rm Cut}(p)$, 
 we have $s(p,q)=\int_0^{d(p,q)}e^{-2\frac{\; \phi(\gamma_t)\;}{n-m}}\d t=L^h(\gamma)$.  
Combining this with 
 $s(p,q)=d^{\,h}(p,q)$ we get $d^{\,h}(p,q)=L^h(\gamma)$. 
Then $\gamma$ is a minimal geodesic in the $h$ metric. In particular, $\nabla^h_{\frac{\d\gamma}{\d s}}\frac{\d \gamma}{\d s}=0$. 
Applying the formula for connection of $h$ in terms of $g$, we have 
\begin{align*}
0&=\nabla^h_{\frac{\d\gamma}{\d s}}\frac{\d \gamma}{\d s}\\
&=\nabla^g_{\frac{\d\gamma}{\d s}}
\frac{\d \gamma}{\d s}
-\frac{4}{n-m}\left\langle \frac{\d \gamma}{\d s},  \nabla \phi \right\rangle \frac{\d \gamma}{\d s}+\frac{2}{n-m}
\left\langle \frac{\d \gamma}{\d s},\frac{\d \gamma}{\d s}\right\rangle  \nabla\phi\\
&=
\frac{2e^{\;\frac{4 \phi(\gamma_r)\;}{n-m}}}{n-m}
\left(-\langle \dot{\gamma}_r, \nabla\phi \rangle \dot{\gamma}_r+
 \nabla\phi\right). 
\end{align*}
Then we obtain that $ \nabla\phi=\langle  \nabla\phi,\dot{\gamma}_r\rangle \dot{\gamma}_r$, i.e., 
$ \nabla\phi$ is parallel to $\dot{\gamma}$. 
Suppose further that $s(p,x)=d^{\,h}(p,x)$ for any $x\in M$. 
Let $x_1,x_2\in M$ be the points 
in the sphere $\partial B_r(p)$ for $r>0$ and 
$c:[0,1]\to \partial B_r(p)$ a curve on $ \partial B_r(p)$ 
joining $c(0)=x_1$ and $c(1)=x_2$. Then we see $\langle \nabla\phi,\dot{c}_t\rangle =0$, because $\nabla\phi$ is parallel to $\dot{\gamma}$, where $\gamma$ is the $g$-geodesic from $p$ to a point in ${\rm Im}(c)$.  
Hence $\phi(x_2)-\phi(x_1)=\int_0^1\langle \nabla\phi,\dot{c}_t\rangle \d t=0$.  
\end{proof}

Here we encounter that $s$ 
does not necessarily satisfy the triangle inequality. 
To get around this difficulty  we  utilize again the conformal metric  $h$.

From $d^{\,h}(p,x)\leq s(p,x)$ and the triangle inequality for the $h$-metric we have 
\begin{align*}
s(p,x)+s(q,x)\geq d^{\,h}(p,x)+d^{\,h}(q,x)\geq d^{\,h}(p,q).
\end{align*}
\begin{proof}[Proof of Theorem~\ref{thm:ChengDiamSphere}]
First note that $\s_{\kappa}(s)=\s_{\kappa}(\delta_{\kappa}-s)$ 
holds for $s\in[0,\delta_{\kappa}]$ by Lemma~\ref{lem:SymJacob}. In particular, we have 
$\cot_{\kappa}(s)=-\cot_{\kappa}(\delta_{\kappa}-s)$ 
for all $s\in[0,\delta_{\kappa}]$. 
Let $r_p$ and $r_q$ be the distance functions to $p$ and $q$ respectively. Then by Theorem~\ref{thm:GlobalLapComp}, we have 
\begin{align*}
\Delta_{ \nabla\phi}(r_p+r_q)(x)\leq (n-m)e^{-\frac{\;2 \phi(x)\;}{n-m}}\left(\cot_{\kappa}(s_p(x))+\cot_{\kappa}(s_q(x)) \right)
\end{align*}
holds in the barrier sense. 
We also have $s_p(x)+s_q(x)\geq d^{\,h}(p,q)=\delta_{\kappa}$, so that 
\begin{align*}
\cot_{\kappa}\left(s_q(x) \right)\leq\cot_{\kappa}\left(\delta_{\kappa}-s_p(x) \right)=-\cot_{\kappa}\left(s_p(x) \right).
\end{align*}
Thus, $\Delta_{\nabla\phi}(r_p+r_q)\leq0$ holds in the barrier sense. 
Note that $\inf_M(r_p+r_q)$ attains its minimum at a point of minimal geodesic joining $p$ and $q$. Then one can apply the 
strong minimum principle for superharmonic functions in the barrier sense (see \cite{Calabi:strongmax, Esc-Hein} for the strong maximum principle for subharmonic functions in the barrier sense) so that 
$r_p(x)+r_q(x)=d(p,q)$ for all $x\in M$ and all geodesics starting point at $p$ in $M$ are minimizing and end at $q$. 
In particular, we have $\Delta_{\nabla\phi}(r_p+r_q)=0$ in the classical sense. Therefore, we have 
\begin{align*}
\cot_{\kappa}(s_p(x))=\cot_{\kappa}(\delta_{\kappa}-s_q(x)) \quad\text{ for all }\quad x\in M.  
\end{align*}    
Since $s\mapsto\cot_{\kappa}(s)$ is strictly decreasing, we have  $s_p(x)+s_q(x)=\delta_{\kappa}$. Hence 
$s_p(x)+s_q(x)=d^{\,h}(p,q)=s(p,q)=\delta_{\kappa}$ 
by $d^{\,h}(p,q)\leq s(p,q)\leq\delta_{\kappa}$ (see Theorem~\ref{thm:WeightedMyers}). We can apply the similar 
argument so that  
$d_p^{\,h}(x)+d_q^{\,h}(x)=d^{\,h}(p,q)=s(p,q)=\delta_{\kappa}$. 
Hence  
$0\leq s_p(x)-d_p^{\,h}(x)=d_q^{\,h}(x)-s_q(x)\leq0$ implies
$s_p(x)=d_p^{\,h}(x)$. Taking $x\notin{\rm Cut}(p)$, we see that 
there exists a unique minimal unit speed geodesic $\gamma$ with $\gamma_0=p$ and $\gamma_{r_p(x)}=x$ satisfying 
$s_p(x)=\int_0^{r_p(x)}e^{-\frac{\;2 \phi(\gamma_t)\;}{n-m}}\d t$. 
 Applying this with Lemma~\ref{lem:Vindepenedent}, 
$\phi$ is rotationally symmetric around $p$.
Secondly, we can deduce that 
\begin{align}
\Delta_{\nabla\phi}r_p(x)&=(n-m)e^{-\frac{\;2 \phi(x)\;}{n-m}}
\cot_{\kappa}(s_p(x)),\label{eq:Lap1}\\
\Delta_{\nabla\phi}r_q(x)&=(n-m)e^{-\frac{\;2 \phi(x)\;}{n-m}}
\cot_{\kappa}(s_q(x))\label{eq:Lap2}
\end{align}
hold in the barrier sense respectively. Consequently, \eqref{eq:Lap1} (resp.~\eqref{eq:Lap2}) holds for $x\in (\text{\rm Cut}(p)\cup\{p\})^c$ (resp.~$x\in (\text{\rm Cut}(q)\cup\{q\})^c$). 
Let $\eta$ be a minimal unit speed geodesic from $p$ to $q$ with $\dot{\eta}_0=\theta$. 
Applying Lemma~\ref{lem:LaplacianComparisonconformal} to \eqref{eq:Lap1}, 
we obtain $m=1$ and the expression of a metric of the form
\begin{align*}
g_{\eta_r}&=\d r^2+e^{\frac{\;2(\phi(r)+\phi(0))\;}{n-1}} \s_{\kappa}^2(s(r))g_{\,\mathbb{S}^{n-1}},\quad 0\leq r\leq d(p,q) 
\end{align*} 
with $s(r)=\int_0^r e^{-\frac{\;2 \phi(t)\;}{n-1}}\d t$ 
and $s(d(p,q))=\delta_{\kappa}$. 
This implies the conclusion. 
\end{proof}

\section{Proof of Theorem~\ref{thm:Splitting}}
Let $\gamma$ be a ray in $M$, i.e. a unit speed geodesic defined on $[0,+\infty[$ such that 
$d(\gamma_t,\gamma_s)=|s-t|$ for any $s,t\geq0$. The \emph{Busemann function} $b_{\gamma}:M\to\R$ for a ray $\gamma$ is defined by 
\begin{align*}
b_{\gamma}(x):=\lim_{t\to+\infty}\left(t-d(x,\gamma_t) \right), \quad x\in M.
\end{align*}
It follows from the triangle inequality that $t\mapsto d(x,\gamma_t)$ is monotonically 
non-decreasing in $t$, so that the above limit  exists. Moreover, it is well-known that $b_{\gamma}$ is a $1$-Lipschitz function.  See e.g. \cite{SchoenYau:LectDiffGeo}.
\begin{lemma}\label{lem:LaplacianCompAlongGeo}
Let $(M,g)$ be an $n$-dimensional complete Riemannian manifold and $V$ a $C^1$-vector field. 
Fix a point $p\in M$. 
Suppose that \eqref{eq:RiciLowBdd}
holds for any $x\in M$ with $\kappa\equiv0$.
Let $q\in M$ be a point such that $r_p$ is smooth at $q$, 
and let $\gamma$ be the unique unit speed 
minimal geodesic from $p$ to $q$. Then we have 
\begin{align}
(\Delta_V r_p)(q)\leq \frac{n-m}{\exp\left(\frac{2V_{\gamma}(r_p(q))}{n-m} \right)\int_0^{r_p(q)}\exp\left(-\frac{2V_{\gamma}(s)}{n-m} \right)\d s}.\label{eq:LaplacianComparison0}
\end{align} 
\end{lemma}
\begin{proof}
Applying the Riccati inequality \eqref{eq:lambda/dr} along $\gamma$ under \eqref{eq:RiciLowBdd} 
with $\kappa\equiv0$, we see 
\begin{align*}
\frac{1}{\lambda(r)^2}\frac{\d\lambda}{\d r}(r)\leq -
\frac{C_p}{n-m}e^{-\frac{2V_{\gamma}(r)}{n-m}}. 
\end{align*}
Integrating this from $\eps>0$ to $r_p(q)$ and letting $\eps\to0$, we have from $\lim_{\eps\to0}\lambda(\eps)=+\infty$ that 
\begin{align*}
\lambda(r_p(q))=C_p^{-1}
e^{\frac{2V_{\gamma}(r_p(q))}{n-m}}(\Delta_Vr_p)(q)
\leq \frac{n-m}{C_p\int_0^{r_p(q)}e^{-\frac{2V_{\gamma}(r)}{n-m}}\d r}.
\end{align*}
This implies the conclusion. 
\end{proof}

\begin{lemma}\label{lem:subharmonic}
Let $(M,g)$ be an $n$-dimensional complete Riemannian manifold and $V$ a $C^1$-vector field. 
Suppose that $(M,g,V)$ is $(V,m)$-complete. 
Suppose that 
\eqref{eq:RiciLowBddStrong}
holds for any $p,x\in M$
with $\kappa=0$.
Then the Busemann function $b_{\gamma}$ for any ray $\gamma$ in $M$ is an $\Delta_V$-subharmonic function in the barrier sense, 
i.e., for each  $p\in M$ and any $\eps>0$, there exists a smooth function $b_{p,\eps}$ defined on a neighborhood $U_{\eps}(p)$ at $p$ such that $b_{p,\eps}(p)=b_{\gamma}(p)$, $b_{p,\eps}\leq b_{\gamma}$ on $U_{\eps}(p)$, and 
$\Delta_Vb_{p,\eps}(p)\geq -\eps$. 
\end{lemma}
\begin{proof}
Fix $p\in M$ and a ray $\gamma$ in $M$. Take any sequence $\{t_k\}$ 
satisfying $\lim_{k\to\infty}t_k=+\infty$. 
Let $\eta_{t_k}$ be a minimal $g$-geodesic joining $p$ and $\gamma_{t_k}$. 
As stated in \cite{Esc-Hein}, there exists a subsequence of $t_k$ such that the initial vector 
$\dot{\eta}_{t_k}(0)$ converges to some unit vector $u\in T_pM$. Let $\eta$ be the ray emanating from $p$ and generated by $u$. Then $p$ does not belong to the cut-locus of $\eta(r)$,  hence $\eta(r)\notin {\rm Cut}(p)$ for any $r>0$. So $b_{\gamma}^r(x):=r-d(x,\eta(r))+b_{\gamma}(p)$ is smooth around $p$ and satisfies $b_{\gamma}^r\leq b_{\gamma}$ with 
$b_{\gamma}^r(p)= b_{\gamma}(p)$. By \eqref{eq:LaplacianComparison0}, we see that for the unique  unit speed geodesic $\overline{\gamma}$ 
from $\eta(r)$ to $p$
\begin{align}
\Delta_Vb_{\gamma}^r(p)=-\Delta_Vr_{\eta(r)}(p)&\geq 
-\frac{n-m}{
\exp\left(\frac{2V_{\overline{\gamma}}(d(\eta(r),p))}{n-m} \right)
\int_0^{d(\eta(r),p)}
\exp\left(-\frac{2V_{\overline{\gamma}}(t)}{n-m}\right)\d t}.\label{eq:subharmbarrier}
\end{align}
Note that $\eta(r)=\overline{\gamma}_{d(p,\eta(r))-r}$ for 
$r\in[0,d(p,\eta(r))]$. 
Then \eqref{eq:subharmbarrier} becomes 
\begin{align}
\Delta_Vb_{\gamma}^r(p)=-\Delta_Vr_{\eta(r)}(p)\geq -\frac{n-m}{\int_0^{d(p,\eta(r))}\exp\left(-\frac{2V_{\eta}(u)}{n-m} \right)\d u}.\label{eq:subharmbarrier*}
\end{align}
Since $(M,g,V)$ is $(V,m)$-complete, we can construct the desired support function. 
\end{proof}

\begin{proof}[Proof of Theorem~\ref{thm:Splitting}]
Let $\gamma:]-\infty,+\infty[\to M$ be a line (i.e., $d(\gamma_t,\gamma_s)=|s-t|$ for 
$s,t\in\R$) and $\gamma^+,\gamma^{-}$ rays defined by $\gamma^+_t:=\gamma_t$ and $\gamma^{-}_t:=\gamma_{-t}$ ($t\geq0$). Let $b^+$, $b^-$ be the Busemann function 
associated to $\gamma^+$, $\gamma^-$, respectively. Then, under the $(V,m)$-completeness of $(M,g,V)$,   
 $b^+$ and $b^-$ 
are continuous $\Delta_V$-subharmonic functions on $M$ in the barrier sense by Lemma~\ref{lem:subharmonic}.  Since $\gamma$ is a line, for each $x\in M$, we have 
\begin{align*}
b^+(x)+b^-(x)=\lim_{t\to+\infty}(2t-d(x,\gamma_t)-d(x,\gamma_{-t}))\leq0
\end{align*}
 and $b^++b^-=0$ on $\gamma$. 
 In view of the strong maximum principle for 
 $\Delta_V$-subharmonic functions in the barrier sense (see \cite{Calabi:strongmax, Esc-Hein} and \cite[Lemma~2.4]{FanLiZhang}), 
 we have $b^++b^-=0$ on $M$. In particular, $b^+$ and $b^-$  are continuous 
 $\Delta_V$-harmonic functions in the barrier sense. Since $|\nabla r_p|=1$ on $({\rm Cut}(p)\cup \{p\})^c$, we have 
 $|\nabla b^+|=|\nabla b^-|=1$ on $M$.  
 Moreover, let $h^{\pm}$ be the smooth $\Delta_V$-harmonic function on an open ball $B$ such that $b^{\pm}=h^{\pm}$ on $\partial B$. Applying the weak  maximum principle to the $\Delta_V$-harmonic function $b^{\pm}-h^{\pm}$ on $B$ 
 in the barrier sense, we can  deduce $b^{\pm}\leq h^{\pm}$ on $B$, hence $0=b^++b^-\leq h^++h^-$. Applying the strong maximum principle again to the smooth $\Delta_V$-harmonic function $h^++h^-$ on $B$, we have $h^++h^-\equiv0$ on $B$.  
 Thus, we can get $0\geq b^+-h^+=-(b^--h^-)\geq0$ on $B$, hence $b^{\pm}=h^{\pm}$ on $B$. Therefore, $b^{\pm}$ is smooth on any ball $B$, hence on $M$. 
Applying \cite[Lemma~6.5]{Wylie:WarpedSplitting} to the 
 smooth $\Delta_V$-harmonic function $b_{\gamma^{\pm}}$ and $|\nabla b_{\gamma^{\pm}}|=1$ on $M$, we can deduce that  
 ${\rm Ric}_{1,n}(\Delta_V)(\nabla b_{\gamma^{\pm}},\nabla b_{\gamma^{\pm}})=0$ and $n-1$ non-zero eigenvalues of 
 ${\rm Hess}\,b^{\pm}|_p$ are all equal, because 
 ${\rm Hess}\,b^{\pm}|_p$ has $n-1$ non-zero eigenvalues. 
 Applying \cite[Lemma~6.6]{Wylie:WarpedSplitting} to the 
 smooth $\Delta_V$-harmonic function $b^{\pm}$ satisfying 
 $|\nabla b^{\pm}|=1$  
  together with the fact that  
${\rm CD}(0,m)$-condition implies ${\rm CD}(0,1)$-condition for $m<1$, we have that  
 $g$ has a twisted product of the form $g=dr^2+e^{\frac{\;2\phi\;}{n-1}}g_N$, 
where $g_N$ is a metric on $N$ and $\phi:M\to\R$ is a smooth function, ${\rm Ric}_{1,n}(\Delta_V)\left(\nabla b^{\pm},\nabla b^{\pm} \right)=0$, and $V=\frac{\partial \phi}{\partial r}\cdot\frac{\partial}{\partial r}+U$ with $U\perp \frac{\partial}{\partial r}$. 
 In the same way of the proof of \cite[Corollary~6.7]{Wylie:WarpedSplitting}, we can deduce that $\frac{d\phi}{dr}=0$, because \cite[Proposition~2.1]{Wylie:WarpedSplitting} yields ${\rm Ric}_{1,n}(\Delta_V)\left(\frac{\partial}{\partial r}, \frac{\partial}{\partial r}\right)=0$ and 
\begin{align*}
0\leq{\rm Ric}_{m,n}(\Delta_V)\left(\frac{\partial}{\partial r}, \frac{\partial}{\partial r}\right)&={\rm Ric}_{1,n}(\Delta_V)
\left(\frac{\partial}{\partial r}, \frac{\partial}{\partial r}\right)+\left(\frac{m-1}{(n-1)(n-m)} \right)\left(\frac{d\phi}{dr} \right)^2\\
&=\left(\frac{m-1}{(n-1)(n-m)} \right)\left(\frac{d\phi}{dr} \right)^2\leq0.
\end{align*}
This means that $g$ has the form of product metric  $g=dr^2+e^{\frac{2\phi(0,\cdot)}{n-1}}g_N=dr^2+h_N$ on $\R\times N$. 
Moreover, we can see that $V$ is a vector field on $N$ by using the fact that ${\rm Ric}_{m,n}(\Delta_V)\left(\frac{\partial}{\partial r},U \right)=0$ for all $U\perp \frac{\partial}{\partial r}$. 
\end{proof}


\providecommand{\bysame}{\leavevmode\hbox to3em{\hrulefill}\thinspace}
\providecommand{\MR}{\relax\ifhmode\unskip\space\fi MR }
\providecommand{\MRhref}[2]{%
  \href{http://www.ams.org/mathscinet-getitem?mr=#1}{#2}
}
\providecommand{\href}[2]{#2}


\begin{thebibliography}{99}

\bibitem{Ambrose}
\textsc{W.~Ambrose}, 
{A theorem of Myers},  
Duke Math. J. {\bf 24} (1957), no.~3, 345--348. 


\bibitem{AN}
\textsc{B. Andrews and L. Ni}, {Eigenvalue comparison on Bakry-Emery manifolds}, Comm. Partial Differential Equations. {\bf 37} (2012), no.~11, 2081--2092.

  

\bibitem{BakryLect1581}
\textsc{D.~Bakry}, {L'hypercontractivit\'e et son utilisation en th\'eorie des semigroupes}, in: Lecture Notes in Math., vol. 1581, Springer-Verlag, Berlin/New York, 1994, pp. 1--114.

\bibitem{BGL_book}
\textsc{D.~Bakry, I.~Gentil and M.~Ledoux},
{Analysis and geometry of Markov diffusion operators}.
Grundlehren der Mathematischen Wissenschaften {\bf 348}, Springer, Cham, 2014.

\bibitem{BE1}
\textsc{D.~Bakry and M.~\'Emery}, 
        {Diffusion hypercontractives}, in: S\'em. Prob. XIX, in: Lecture Notes in Math., vol. 1123, Springer-Verlag, Berlin/New York, 1985, pp. 177--206.

\bibitem{BL} 
\textsc{D. Bakry and M. Ledoux}, {A logarithmic Sobolev form of the Li-Yau parabolic inequality}, Rev. Mat. Iberoam. {\bf 22} (2006), no.~2, 683--702.

\bibitem{BQ1}
\textsc{D.~Bakry and Z.-M.~Qian},  {Harnack inequalities on a manifold with positive or negative Ricci curvature}, Rev. Mat. Iberoam. {\bf 15} (1999), no.~1, 143--179.

\bibitem{BQ2}
\textsc{D.~Bakry and Z.-M.~Qian}, 
{Volume comparison theorems without Jacobi fields}, Current trends in potential theory, 115--122, Theta Ser. Adv. Math., {\bf 4}, Theta, Bucharest, 2005.


   \bibitem{Calabi:strongmax}
   \textsc{E.~Calabi},
  {An extension of E. Hopf's maximum principle with an application to Riemannian geometry},   
   Duke Math.~J. {\bf 25} (1957), no.~1, 45--56.

\bibitem{CavOliSantos}
\textsc{M.~P.~Cavalcante, J.~Q.~Oliveira and M.~S.~Santos}, 
{Compactness in weighted manifolds and applications},  
Results in Math. {\bf 68} (2015), no.~1-2, 143--156. 

\bibitem{Esc-Hein}
       \textsc{J.-H.~Eschenburg and E.~Heintze}, 
     {An elementary proof of the Cheeger-Gromoll splitting theorem},  
     Ann.~Global.~Anal.~and Geom. {\bf 2} (1984), no.~2, 141--151.


\bibitem{FanLiZhang}
\textsc{F.~Fang, X.-D.~Li and Z.~Zhang}, 
{Two generalizations of Cheeger-Gromoll splitting theorem via Bakry-Emery Ricci curvature}, Ann. Inst. Fourier (Grenoble) {\bf 59} (2009), no. 2, 563--573.

\bibitem{FLL}  
\textsc{A.~Futaki, H.-Z.~Li and  X.-D.~Li}, 
{On the first eigenvalue of the Witten Laplacian and the diameter of compact shrinking Ricci solitons}, Ann. Global Anal. Geom. {\bf 44} (2013), no.~2, 105--114.

\bibitem{Hsu:2001}
\textsc{E.~P.~Hsu}, 
{Stochastic analysis on manifolds}, 
 Graduate Studies in Mathematics, {\bf 38}. American Mathematical Society, Providence, RI, 2002.  




\bibitem{KolesMilman}
\textsc{A.~V.~Kolesnikov and E.~Milman}, 
{Poincar\'e and Brunn-Minkowski inequalities on weighted Riemannian manifolds with boundary}, 
Amer. J. Math. {\bf 140} (2018), no.~5, 1147--1185.

\bibitem{KL}
\textsc{K.~Kuwae and X.-D.~Li},
{New Laplacian comparison theorem 
and its applications to diffusion processes 
on Riemannian manifolds}, 
to appear in Bulletin of London Math. Soc., Available from {\tt arXiv:2001.00444}.



\bibitem{LL15} 
\textsc{S. Li and  X.-D. Li}, 
{W-entropy formula for the Witten Laplacian on manifolds with time dependent metrics and potentials}, Pacific J. Math. {\bf 278} (2015), no.~1, 173--199.

\bibitem{LL17}  
\textsc{S. Li and  X.-D. Li},  
{Hamilton differential Harnack inequality and W-entropy for Witten} 
{Laplacian on Riemannian manifolds}, J. Funct. Anal. 
 {\bf 274} (2018), no. 11, 3263--3290. 

 \bibitem{Xdli:Liouville}
 \textsc{X.-D.~Li}, 
 {Liouville theorems for symmetric diffusion operators on complete Riemannian manifolds},  J. Math. Pures Appl. (9) {\bf 84} (2005), no.~10, 1295--1361.


\bibitem{Li12} 
\textsc{X.-D.~Li},
{Perelman's entropy formula for the Witten Laplacian on Riemannian manifolds via Bakry-Emery Ricci curvature}, Math. Ann. {\bf 353} (2012), no.~2, 403--437.


\bibitem{Lot}  
\textsc{J. Lott}, 
{Some geometric properties of the Bakry-\'Emery Ricci tensor}, Comment. Math. Helv. {\bf 78} (2003), no.~4, 
865--883.

\bibitem{LV2}
\textsc{J.~Lott and C.~Villani},
{Ricci curvature for metric-measure spaces via optimal transport}, 
Ann.\ of Math.\ {\bf 169} (2009), no.~3, 903--991.


\bibitem{Milman17}
\textsc{E.~Milman}, 
{Beyond traditional Curvature-Dimension I: new model spaces for isoperimetric and concentration inequalities in negative dimension}, 
Trans. Amer. Math. Soc. {\bf 369} (2017), no.~5, 3605--3637. 



\bibitem{Ohta:KN}
\textsc{S.~Ohta}, 
{$(K,N)$-convexity and the curvature-dimension condition for negative $N$}, 
 J. Geom. Anal. {\bf 26} (2016), no.~3, 2067--2096.
 
 \bibitem{OhtaTaka}
\textsc{S.~Ohta and A.~Takatsu},
 {Displacement convexity of generalized relative entropies}, 
 Adv. Math. {\bf 228} (2011), no.~3, 1742--1787.
 

\bibitem{Pet:RiemannianGeo}
\textsc{P.~Petersen}, 
{Riemannian geometry}, second edition, Graduate Texts in Mathematics, vol. {\bf 171}, Springer, New York, 2006. 

\bibitem{Qian}  
\textsc{Z.-M. Qian}, 
{Estimates for weight volumes and applications}, 
Quarterly J. Math. 
{\bf 48} (1987), no.~2, 235--242.


\bibitem{SchoenYau:LectDiffGeo}
\textsc{R.~Schoen and S.~T.~Yau}, 
 {Lectures on differential geometry}, 
 Conference Proceedings and Lecture Notes in Geometry and Topology, I. International Press, Cambridge, MA, 1994.

\bibitem{St:geomI}
\textsc{K.-T. Sturm}, {On the geometry of metric measure spaces. {I}}, Acta Math.
  \textbf{196} (2006), no.~1, 65--131.

\bibitem{St:geomII}
\textsc{K.-T. Sturm}, {On the geometry of metric measure spaces. {II}}, Acta Math.
 \textbf{196} (2006), no.~1, 133--177.


\bibitem{Tadano}
\textsc{H.~Tadano},  
{Some Ambrose and Galloway types theorems via Bakry-\'Emery and modified Ricci curvatures},  Pacific J. Math. {\bf 294} (2018), no.~1, 213--231.

\bibitem{TadanoNegative}
\textsc{H.~Tadano}, 
{Some compactness theorems via $m$-Bakry-\'Emery and $m$-modified Ricci curvatures with negative $m$}, Differential Geometry and its Applications, {\bf 75} (2021), 101720. 
 


\bibitem{WeiWylie}
\textsc{G.~Wei and W.~Wylie}, 
{Comparison geometry for the Bakry-Emery Ricci tensor}, 
J. Differential Geom. {\bf 83} (2009), no. 2, 377--405.


\bibitem{Wylie:WarpedSplitting}
\textsc{W.~Wylie},
{A warped product version of the Cheeger-Gromoll splitting theorem}, 
Trans. Amer. Math. Soc. {\bf 369} (2017), no.~9, 6661--6681. 

\bibitem{WylieYeroshkin}
\textsc{W.~Wylie and D.~Yeroshkin}, 
{On the geometry of Riemannian manifolds with density}, 
preprint, 2016. 


\bibitem{Zhu97}
{S.~Zhu}, 
{The Comparison Geometry of Ricci Curvature}, 
Comparison Geometry MSRI Publications {\bf 30}, (1997), 221--262. 




\end{thebibliography}
\end{document}